\documentclass[10pt]{article}
\usepackage{amsthm}
\usepackage{epsfig}
\usepackage{amsmath}
\usepackage{graphicx}
\usepackage{amsfonts}
\usepackage{ulem}
\usepackage[authoryear]{natbib}
 \usepackage{geometry}
 \usepackage{cancel}	
\usepackage{color}
\usepackage{hyperref}
\usepackage{setspace}
\usepackage{cite}

\usepackage{lipsum}

\newtheorem{theorem}{Theorem}[section]

\newtheorem{Proposition}[theorem]{Proposition}

\theoremstyle{definition}
\newtheorem{definition}[theorem]{Definition}

\theoremstyle{remark}
\newtheorem{remark}{Remark}
\newtheorem{example}[theorem]{Example}

\numberwithin{equation}{section}	 
\newcommand{\YB}[1]{{\color{black} #1}}

\title{ }

\RequirePackage{filecontents}

\begin{document}
\begin{flushleft}\huge 
\textbf{Asymptotic Properties of the Empirical Spatial Extremogram}
\end{flushleft}
\vspace*{1ex}
\begin{flushleft}\Large
\textbf{Yongbum Cho and Richard A. Davis}\\
\large
Department of Statistics, Columbia University\\
\Large
\textbf{Souvik Ghosh}\\
\large
LinkedIn Corporation
\normalsize
\end{flushleft}






\begin{flushleft}
\vspace*{2ex}
\textbf{ABSTRACT.} The extremogram is a useful tool for measuring extremal dependence and checking model adequacy in a time series.  We define the extremogram in the spatial domain when the data is observed on a lattice or at locations distributed as a Poisson point process in $d$-dimensional space. We establish a central limit theorem for the empirical spatial extremogram. We show these conditions are applicable for max-moving average processes and Brown-Resnick processes and illustrate the empirical extremogram's performance via simulation. We also demonstrate its practical use with a data set related to rainfall in a region in Florida.\\
\vspace*{2ex}
\textit{Keywords:} extremal dependence; extremogram; max moving average; max stable process;  spatial dependence
\end{flushleft}

\section{Introduction} 
Extreme events can affect our lives in many dimensions. Events like large swings in financial markets or extreme weather conditions such as floods and hurricanes can cause large financial/property losses and numerous casualties. Extreme events  often appear to cluster and that has resulted in a growing interest in  measuring \textit{extremal dependence}  in many areas including finance, insurance, and atmospheric science. 

{Extremal dependence} between two random vectors $X$ and $Y$ can be viewed as the probability that $X$ is extreme given 
$Y$ belongs to an extreme set. 
The \textit{extremogram}, proposed by \citet{Richard09}, is a versatile tool for assessing extremal dependence in a stationary time series. 
The extremogram has two main features:
\begin{itemize} 
\item It can be viewed as the extreme-value analog of the autocorrelation function of a stationary time series, i.e., extremal dependence is expressed as a function of lag. 
\item It allows for measuring dependence between random variables belonging in a large variety of extremal sets.   Depending on choices of sets, many of the commonly used extremal dependence measures - right tail dependence, left tail dependence, or dependence among large absolute values - can be treated as a special case of the extremogram. The flexibility coming from arbitrary choices of extreme sets have made it especially well suited for time series applications such as high-frequency FX rates (\citet{Richard09}), cross-sectional stock indices (\citet{Richard10}), and CDS spreads (\citet{CDS(extremogram)}). 
\end{itemize}

In this paper, we will define the notion of the extremogram for random fields defined on $\mathbb{R}^d$ for some  $d>1$ and investigate the asymptotic properties of its corresponding empirical estimate. Let $\{ X_s,  s \in  \mathbb{R}^d\} $ be a stationary $\mathbb{R}^k$-valued random field. For measurable sets  $A,B \subset \mathbb{R}^k$ bounded away from \textbf{0}, we define the \textit{spatial extremogram} as
\begin{gather}
\label{ee.def.1.1}	 \rho_{AB}(h) = \lim_{x \rightarrow \infty} P(X_h \in xB| X_{\textbf{0}} \in xA),  \quad h \in \mathbb{R}^d,
\end{gather}
provided the limit exists. We call (\ref{ee.def.1.1}) the \textit{spatial extremogram} to emphasize that it is for a random field in $\mathbb{R}^d.$  If one takes $A=B=(1, \infty)$ in the $k=1$ case, then we recover the tail dependence coefficient between $X_h$ and $X_{\textbf{0}}.$ For light tailed time series, such as stationary Gaussian processes, $\rho_{AB}(h) = 0$ for $h \ne \textbf{0}$ in which case there is no extremal dependence. However, for heavy tailed processes in either time or space, $\rho_{AB}(h) $ is often non-zero for many lags $h \neq \textbf{0}$ and for most choices of sets $A$ and $B$ bounded away from the origin. 

We will consider estimates of $\rho_{AB}(h)$ under two different sampling scenarios. In the first, observations are taken on the lattice $\mathbb{Z}^d$.  Analogous to \citet{Richard09}, we define the \textit{empirical spatial extremogram} (ESE)  as
        \begin{equation}
		\label{ee.lattice}  \hat{\rho}_{AB {,m}}(h) = \frac{ \sum_{s, t \in \Lambda_n, s-t=h}   I_ { \{a_m^{-1} X_s \in A,  a_m^{-1} X_{t} \in B \} } / n(h)}  {  \sum_{s \in \Lambda_n}  I_ { \{ a_m^{-1} X_s \in A \}}/  {  \# \Lambda_n} },
	\end{equation}
where 
\begin{itemize}
\item $\Lambda_n=\{1,2,.\ldots,n\}^d$ is the $d$-dimensional cube with side length $n$, 
\item $h\in \mathbb{Z}^d$ are observed lags in $\Lambda_n$,
\item $m=m_n$ is an increasing sequence satisfying $m\to \infty$ and  $ m/n \rightarrow 0$ as $n\to \infty$,
\item $ a_m$ is a sequence such that $ P(|X| > a_m) \sim m^{-1}$,
\item $n(h)$ is the number of pairs in $\Lambda_n$ with lag $h$, and
\item $  \#  \Lambda_n  $ is the cardinality of $\Lambda_n$.
\end{itemize}

 {In the second case, the data are assumed to come from a stationary random field $X_s,$ where the locations $\{s_1,...,s_{N} \}$ are assumed to be points of a homogeneous Poisson point process on $S_n\subset \mathbb{R}^d$.  We define the {empirical spatial extremogram} as a kernel estimator of $\rho_{AB}(h)$, in the spirit of the estimate of autocorrelation in space (see \citet{lgs}). Under suitable growth conditions on $S_n$ and restrictions on the kernel function, we show that the weighted estimator of $\rho_{AB}(h)$ is consistent and asymptotically normal.} 

The organization of the paper is as follows: In Section \ref{section.main.result}, we present the asymptotic properties of the ESE for both cases described above. Section \ref{section.examples} provides examples illustrating the results of Section \ref{section.main.result} together with a simulation study demonstrating the performance of the ESE. In Section \ref{section.application}, the spatial extremogram is applied to a spatial rainfall data set in Florida. The proofs of all the results are in the Appendix.

\section{Asymptotics of the ESE}  \label{section.main.result}

\subsection{Definitions and notation} \label{section.notation}

 {Let $\{X_s, s\in I\}$ be  a $k$-dimensional strictly stationary random process where $I$ is either $\mathbb{R}^d$ or $\mathbb{Z}^d$.  For $H=\{h_1,\ldots,h_t\}\subset I$, we use $X_{H}$ to denote $(X_{h_1},\ldots,X_{h_t})$. {The random field  is said to be \textit{regularly varying} with \textit{index} $\alpha>0$ if for any $H,$ the radial part $\|X_{H}\|$ satisfies for all $ y > 0$
	\begin{flushleft}
	(C1) $ \displaystyle \qquad  \qquad \qquad  \qquad \qquad   \qquad \qquad  \frac{P\big(\|X_{H}\| >yx \big)}{ P\big( \|X_{H}\| > x\big)} \rightarrow y^{- \alpha} \; \mbox{ as } x\to \infty, $
	\end{flushleft}
and the angular part $\frac{X_{H}}{\|X_{H}\|}$ is asymptotically independent of the radial part $\|X_{H}\|$ for large values of $\|X_{H}\|$, i.e., there exists a random vector $\Theta_H \in \mathbb{S}^{tk-1}$, the unit sphere in $\mathbb{R}^{t k}$ with respect to $\| \cdot \|$, such that
\begin{flushleft}
	(C2) $ \displaystyle \qquad  \qquad \qquad  \qquad \qquad     P \left( \frac{X_{H}}{\|X_{H}\|} \in \cdot \Big| \|X_{H}\| > x \right ) \xrightarrow{w}  P \left( \Theta_H \in \cdot \right ) \; \mbox{ as } x\to \infty,$
	\end{flushleft}
where $\xrightarrow{w}$ denotes weak convergence. The distribution of $ P \left( \Theta_H \in \cdot \right )$ is called the \textit{spectral measure} of $X_H$.}

An equivalent definition of regular variation is given as follows. There exists a sequence $ a_n \rightarrow \infty, \alpha >0$ and a family of non-null Radon measures $(\mu_{H})$ on the Borel $\sigma$-field of $\bar{\mathbb{R}}^{t k} \setminus \{ \textbf{0}\} $ such that  $nP(a_n^{-1}X_{H} \in \cdot) \xrightarrow{v} \mu_{H}(\cdot)$ for $t \geq 1$, where the limiting measure satisfies $\mu_H(y \cdot) = y^{-\alpha}\mu_H(\cdot)$ for $y >0$ . Here, $ \xrightarrow{v}$ denotes vague convergence. Under the regularly varying assumption, one can show that (\ref{ee.def.1.1}) is well defined. \YB{See Section 6.1  of} \citet{Resnick(HT)} for more details.

\subsection{Random fields on a lattice} 

Let $\{X_s, s \in \mathbb{Z}^d \}$ be a strictly stationary random field and suppose we have observations  $\{ X_{s}, s  \in \Lambda_n = \{1,...,n\}^d \}$. Let $d(\cdot,\cdot)$ be a metric on $\mathbb{Z}^d$. We denote the $\alpha$-mixing coefficient by 
\begin{equation*}
\alpha_{j,k}(r)= \sup \big\{ \alpha\big( \sigma(X_s,s\in S),\sigma(X_s,s\in T)\big): S,T\subset \mathbb{Z}^d,  { \# S \le j, \#T\le k}, d(S,T)\ge r  \big\},
\end{equation*}
where for any two $\sigma$-fields $\mathcal{A}$ and $\mathcal{B}$, $\alpha(\mathcal{A},\mathcal{B})=\sup\{|P(A\cap B)-P(A)P(B)|:A\in \mathcal{A},B\in \mathcal{B} \}$  {and for any $S,T\subset \mathbb{Z}^d$, $d(S,T)= \inf\{d(s,t):s\in S, t\in T\}$}.

In order to study asymptotic properties of (\ref{ee.lattice}), we impose regularly varying and certain mixing conditions on the random field. In particular, we use the big/small block argument: the side length of big blocks, $m_n$, and the distance between big blocks, $r_n$, have to be coordinated in the right fashion. To be precise, we assume the following conditions.  \\
\ \\
(\textbf{M1}) Let $B_{\gamma}$ be the ball of radius $  {\gamma}$ centered at 0, i.e., $B_{\gamma} = \{ s \in \mathbb{Z}^d: d(\textbf{0},s) \leq {\gamma} \},$ and set $c = \# B_{\gamma}$. For a fixed $\gamma$, assume that there exist $m_n, r_n \rightarrow \infty$ with {${m_n^{2+2d}}/{n^d} \rightarrow 0 $}, $ {r_{n}^d}/{m_n} \rightarrow 0 $ such that
 {\begin{eqnarray}
&&\label{cond.2} \displaystyle    \lim_{k \rightarrow \infty} \limsup_{n \rightarrow \infty} m_n \sum_{ l \in \mathbb{Z}^d, k < d(\textbf{0},l) \leq r_n}  
   P \left( \max_{ {s} \in B_{\gamma}} |X_{ {s}}| >  {\epsilon a_m},    \max_{  {s}' \in B_{\gamma}+l }|X_{ {s}'}| > {\epsilon a_m} \right) = 0 \quad \text{for}  \quad \forall \epsilon >0,   \\
&&\label{cond.1}       \lim_{n \rightarrow \infty} m_n \sum_{  l \in \mathbb{Z}^d, r_n < d(\textbf{0}, l)}  \alpha_{c,c}(d(\textbf{0},l) ) = 0,\\
&&\label{cond.000}  \sum_{l \in  \mathbb{Z}^d}   \alpha_{j_1,j_2}( d(\textbf{0}, l) ) < \infty \quad \text{for} \quad 2 c \leq  j_1+ j_2 \leq 4 c,   \\
&&\label{cond.3}   \lim_{n \rightarrow \infty}  n^{d/2} m_n^{1/2}\alpha_{c,c n^d} (m_n) = 0,
\end{eqnarray}}
where $a_m$ satisfies $P(|X| > a_m) \sim \frac{1}{m}.$

Condition (\ref{cond.2})  restricts the joint distributions for exceedance as two sets of points become far apart. Conditions (\ref{cond.1}) - (\ref{cond.3}) impose restrictions on the decaying rate of the mixing functions together with the level of the threshold specified by $m_n$.  These conditions are similar to those in \citet{Bolthausen} and \citet{Richard09}.

As in \citet{Richard09}, the ESE $\hat{\rho}_{AB,m}(h)$ is centered by the \textit{Pre-Asymptotic} (PA) extremogram
\begin{gather}
\label{def.pa.extremo} {\rho}_{AB,m}(h)   =  \frac{{\tau}_{AB,m}(h) }{{p}_m(A)},
\end{gather}
where $ {\tau}_{AB,m}(h) =  m_n P(X_{\textbf{0}} \in a_m A, X_h \in a_m B)$ and $p_m(A) = m_n P( X_{\textbf{0}} \in a_m A)$. Notice that (\ref{def.pa.extremo}) is the ratio of the expected value of the numerator and denominator in (\ref{ee.lattice}).

\label{section.asymptotics}
\begin{theorem}
\label{cor:extremogramclt} Suppose a strictly stationary regularly varying  random field $\{X_{\textbf{s}}, s \in  \mathbb{Z}^d\}$ with index $\alpha>0$ is observed on $\Lambda_n = \{1,...,n\}^d$. {For any finite set of non-zero lags $H$ in $ \mathbb{Z}^d$, assume (\textbf{M1}),  \YB{where $B_{\gamma} \supseteq H$ for some $ {\gamma}$}.  Then}
 {\begin{equation*}
 \sqrt{\frac{n^d}{m_n}}  \Big[ \hat{\rho}_{AB,m}(h)  - \rho_{AB,m}(h) \Big]_{h \in H}  \xrightarrow{d} N(\textbf{0}, \Sigma ),
\end{equation*}}
where the matrix $\Sigma$ in normal distribution is specified in Appendix A.
\end{theorem}

We present the proof of Theorem \ref{cor:extremogramclt} in Appendix A. Examples of heavy-tailed processes satisfying (\textbf{M1}) are presented in Section \ref{section.examples}.

\begin{remark} \label{remark:bias}  In Theorem \ref{cor:extremogramclt}, the pre-asymptotic extremogram $ \rho_{AB,m}(h) $ is replaced by the extremogram $ \rho_{AB}(h) $ if  
\begin{eqnarray}
	\label{bias} \lim_{n \rightarrow \infty} \sqrt{\frac{n^d}{m_n}}| \rho_{AB,m}(h) - \rho_{AB}(h) | = 0, \; \quad for \quad \;  h\in H.
\end{eqnarray}
\end{remark}

\subsection{Random fields on $\mathbb{R}^d$}  
\label{section.rd}
Now consider the case of a random field defined on $\mathbb{R}^d$ and the sampling locations are given by points of a Poisson process. In this case, we adopt the ideas from \citet{karr} and \citet{lgs} and use a kernel estimate of the extremogram. For convenience, we restrict our attention to $\mathbb{R}^2$. The extension to $\mathbb{R}^d (d>1)$ is straightforward, but notationally more complex. 

Let $\{X_s, s \in \mathbb{R}^2 \}$ be a stationary regularly varying random field with index $\alpha > 0$. Suppose $N$ is a homogeneous 2-dimensional Poisson process with intensity parameter $\nu$ and is independent of $X$.  Define $N^{(2)}(ds_1, ds_2) = N(ds_1) N(ds_2) I(s_1 \neq s_2)$. Now consider a sequence of compact and convex sets $S_n \subset \mathbb{R}^2$ with Lebesgue measure $|S_n| \rightarrow \infty$ as $n \rightarrow \infty$. Assume that for each $y \in \mathbb{R}^2$
\begin{eqnarray}
\label{Karr.lemma} \lim_{n \rightarrow \infty} \frac{|S_n \cap (S_n - y)|}{|S_n|} = 1,  
\end{eqnarray}
where $S_n - y= \{ x-y: x \in S_n\}$,
\begin{eqnarray}
\label{lgs.lemma}|S_n| = O(n^2), \qquad |\partial S_n| =  O(n),
\end{eqnarray}
and $\partial  S_n$ denotes the boundary of $S_n$.

The spatial extremogram in (\ref{ee.def.1.1}) is estimated by $ \displaystyle \hat{\rho}_{AB,m}(h) = {\hat{\tau}_{AB,m}(h) }/ {\hat{p}_m(A)},$ where
	\begin{gather}
	 \hat{p}_m(A) = \frac{m_n}{\nu |S_n|} \int_{S_n}  I\left(\frac{X_{s_1}}{a_m} \in A \right) N(d s_1),  \label{def.denom.est.RF} \\
	  \hat{\tau}_{AB,m} (h) = \frac{m_n}{\nu^2} \frac{1}{|S_n|} \int_{S_n} \int_{S_n} w_n(h + s_1 - s_2)  \;  I\left(\frac{X_{s_1}}{a_m} \in A \right)    I\left(\frac{X_{s_2}}{a_m} \in B \right)  N^{(2)}(ds_1, ds_2).  \label{def.num.est.RF} 
	\end{gather}
Here $w_n(\cdot) = \frac{1}{\lambda_n^2}w(\frac{\cdot}{\lambda_n})$ is a sequence of weight functions, where $w(\cdot)$ on $\mathbb{R}^2$ is a positive, bounded, isotropic probability density function and $\lambda_n$ is the bandwidth satisfying $\lambda_n \rightarrow 0$ and $\lambda_n^2|S_n| \rightarrow \infty.$ To establish a central limit theorem for $\hat{\rho}_{AB,m}(h)$, we derive asymptotics of the denominator ${\hat{p}_m(A)}$ and numerator ${\hat{\tau}_{AB,m}(h) }$. In order to show consistency of ${\hat{p}_m(A)}$, we assume the following conditions, which are the non-lattice analogs of (\ref{cond.2}) and (\ref{cond.1}).\\
\ \\
(\textbf{M2}) There exist an increasing sequence $m_n$ and $r_n$ with $m_n=o(n)$ and $ r_n^2 = o(m_n)$ such that
\begin{gather} 
\label{M2} \lim_{k \rightarrow \infty}\limsup_{n \rightarrow \infty} \int_{B[k,r_n]} m_n P(|X_y| > \epsilon a_m , |X_{\textbf{0}}| > \epsilon a_m) dy = 0   \quad \text{for}  \quad \forall \epsilon >0,\\
\label{M1} \lim_{n \rightarrow \infty} \int_{\mathbb{R}^2 \setminus B[{0},r_n)} m_n \alpha_{1,1}(y) dy = 0,\\
\label{M4} \int_{\mathbb{R}^2 } \tau_{AA}(y) dy < \infty,
\end{gather}
where $B[a,b) = \{s: a \leq d(\textbf{0},s) < b, s \in \mathbb{R}^2 \}$ and $ \displaystyle {\tau}_{AA}(y) = \lim_{n \rightarrow \infty}   {\tau}_{AA,m}(y)$.

For a central limit theorem for ${\hat{\tau}_{AB,m}(h) }$, the following conditions are required.\\
\ \\
(\textbf{M3}) Consider a cube $B_n  \subset S_n$ with $|B_n| = O(n^{2 a})$ and $|\partial B_n| =  O(n^{a})$ for $0 < a <1.$ Assume that there exist an increasing sequence $m_n$  with $m_n=o(n^{a})$ and $\lambda_n^2 m_n \rightarrow 0$ such that
\begin{gather} 
 \label{clt.cond.2}  \displaystyle \sup_n E\left\{  \sqrt{\frac{|B_n| \lambda_n^2}{m_n}} \big| \hat{\tau}_{AB,m}(h:B_n) - E\hat{\tau}_{AB,m}(h:B_n)   \big|^{2 + \delta}  \right \} \leq C_{\delta},  \quad \delta > 0, \; C_{\delta} < \infty, 
\end{gather}
where $\hat{\tau}_{AB,m}(h:B_n)$ is the quantity (\ref{def.num.est.RF}) with $S_n$ replaced by $B_n$ on the right-hand side. Further assume
\begin{gather} 
\label{irregular.numerator.integration}    \int_{\mathbb{R}^2} \tau_{AB}(y) dy < \infty, \; \;  \int_{\mathbb{R}^2} \alpha_{2,2} (d(\textbf{0},y)) dy < \infty, 
\end{gather}
and 
\begin{gather} 
 \label{clt.cond.3}  \displaystyle \sup_{l} \frac{\alpha_{l,l}(h)}{l^2} = O(h^{-\epsilon}) \quad for \; some \; \epsilon > 0.
\end{gather}

Lastly, the proof requires some smoothness of the random field. 
\begin{definition}
A stationary regularly varying random field $\{X_s, s \in \mathbb{R}^d \}$ satisfies a \textit{local uniform negligibility condition (LUNC)} if for an increasing sequence $a_n$ satisfying $P(|X| > a_n) \sim \frac{1}{n}$ and for all $\epsilon, \delta > 0$, there exists $ \delta'  >0$ such that 
	\begin{eqnarray}
		\label{note.2} \limsup_n nP \left( \sup_{||s|| < \delta'} \frac{|X_s - X_{\textbf{0}}|}{a_n} > \delta  \right) < \epsilon.
	\end{eqnarray}
\end{definition}

\begin{theorem} \label{theorem:clt:2} Let $\{X_s, s \in \mathbb{R}^2 \}$ be a stationary regularly varying random field with index $\alpha > 0$ satisfying LUNC. Assume $N$ is a homogeneous 2-dimensional Poisson process with intensity parameter $\nu$ and is independent of $X$. Consider a sequence of compact and convex sets $S_n \subset \mathbb{R}^2$ satisfying $|S_n| \rightarrow \infty$ as $n \rightarrow \infty$. 
Assume  conditions (\textbf{M2}) and (\textbf{M3}). Then for any finite set of non-zero lags $H$ in $\mathbb{R}^2$,
\begin{eqnarray}
	\label{final}  \sqrt{\frac{|S_n|\lambda_n^2}{m_n}  } \left[ \hat{\rho}_{AB,m}(h) - {\rho}_{AB,m}(h) \right]_{h\in H}\rightarrow N(\textbf{0}, \Sigma),
\end{eqnarray}
where the matrix $\Sigma$ is \YB{specified in the proof of Theorem \ref{cor:extremogramclt} in Appendix A}.
\end{theorem} 

We present the proof of Theorem \ref{theorem:clt:2} in Appendix B. As in Remark \ref{remark:bias}, ${\rho}_{AB,m}(h)$ can be replaced by ${\rho}_{AB}(h)$ if ${\rho}_{AB,m}(h)$ converges fast enough.\begin{remark} In (\ref{final}), ${\rho}_{AB,m}(h)$ can be replaced by $ {\rho}_{AB}(h)$ if
\begin{gather}
	\label{replace}  \lim_{n \rightarrow \infty}  \sqrt{ \frac{|S_n|\lambda_n^2}{m_n}}  |{\rho}_{AB,m}(h) - {\rho}_{AB}(h)| =  0 \qquad for \; h\in H.
\end{gather}
\end{remark}

\section{Examples} 
\label{section.examples}
Here we provide two max-stable processes to illustrate the results of Section \ref{section.main.result}. For background on
 max-stable processes, see \citet{deHaan(1984)} and \citet{deHaan(2007)}. In order to check conditions, we need the result from \citet{dombry}.
\begin{Proposition} [\citet{dombry}] 
\label{dombry.prop} Suppose $\{X_s, s \in S\}$ is a max-stable random field with unit Fr\'{e}chet marginals. If $S_1$ and $S_2$ are finite or countable disjoint closed subsets of $S$, and $\mathcal{S}_1$ and $\mathcal{S}_2$ are the respective $\sigma$-fields generated by each set, then
\begin{gather}
	\label{dombry} \displaystyle \beta (\mathcal{S}_1,\mathcal{S}_2) \leq 4 \sum_{s_1 \in S_1} \sum_{s_2 \in S_2} \YB{\rho_{(1, \infty)(1, \infty)} (||s_1 - s_2||)} 
\end{gather}
where $\beta (\cdot,\cdot)$ is the $\beta$-mixing coefficient. We refer to Lemma 2 in \citet{christina}.
\end{Proposition}

Notice that (\ref{dombry}) provides the upper bound for $\alpha$-mixing coefficient since $2 \alpha(\mathcal{S}_1, \mathcal{S}_2) \leq \beta(\mathcal{S}_1,\mathcal{S}_2).$ See \citet{Bradley}.

\subsection{Max Moving Average (MMA)} \label{Example_MMA}
Let $\{Z_s, s \in \mathbb{Z}^2\}$ be an iid sequence of unit Fr\'{e}chet random variables. The max-moving average (MMA) process is defined by
\begin{gather}
\label{def.gen.imma}  X_t = \max_{s \in \mathbb{Z}^2} w(s) Z_{t-s},
\end{gather}
where $ w(s) > 0 \; \text{and} \; \sum_{s \in \mathbb{Z}^2} w(s)< \infty$. Note that the summability of $w(\cdot)$ implies the process is well defined. Also, notice that $a_m = O(m)$ since marginal distributions are Fr\'{e}chet. {Consider the Euclidean metric $d(\cdot, \cdot)$ and write $|| l || = d(\textbf{0},l)$ for notational convenience}. With $w(s)=I(||s|| \leq 1)$, the process (\ref{def.gen.imma}) becomes the MMA(1): $ \displaystyle X_t = \max_{||s|| \leq 1} Z_{t-s} $.  Using $ A= B= (1,\infty)$, the extremogram for the MMA(1) is then
\begin{equation} \label{MMA.extremogram}
 \rho_{AB}(h)  = \lim_{n\rightarrow \infty} P(X_{h} > a_{m_n} | X_{ \textbf{0}} > a_{m_n}) =  \left\{ \begin{array}{rll} 
		1, \; & \mbox{if} & ||h|| = 0, \\
		2/5, & \mbox{if} & ||h|| = 1, \sqrt{2}, \\
		1/5, & \mbox{if} & ||h|| = 2, \\
		0, \;  & \mbox{if} &  ||h|| >2.
	\end{array}\right.
\end{equation}
Since the process is \textit{2-dependent}, conditions for Theorem \ref{cor:extremogramclt} are easily checked. 

Figure \ref{figure:mma1} (left) shows $\rho_{AB}(h)$ and $\hat{\rho}_{AB,m}(h)$ from a realization of  MMA(1) generated by \textit{rmaxstab} in the  \textit{SpatialExtremes} package\footnote{http://cran.r-project.org/web/packages/SpatialExtremes/SpatialExtremes.pdf} in R. We use 1600 points ($\Lambda_n = \{1,...,40\}^2 \in \mathbb{Z}^2$) and set $A = B = (1, \infty)$ and $a_m = $ .97 quantile of the process. In the figure, the dots and the bars correspond to ${\rho}_{AB}(h)$ and $\hat{\rho}_{AB,m}(h)$ for observed distances in the sample. \YB{The dashed line corresponds to $0.03 \; ( = 1 - 0.97)$ and} two horizontal lines are 95\% \textit{random permutation} confidence bands to check the existence of extremal dependence (see \citet{Richard10}). The bands suggest ${\rho}_{(1,\infty)(1,\infty),m}(h) = 0$ for $h > 2$, which is consistent with (\ref{MMA.extremogram}).

\begin{figure}[t!]
\centering
 \includegraphics[width=15cm, height = 7cm]{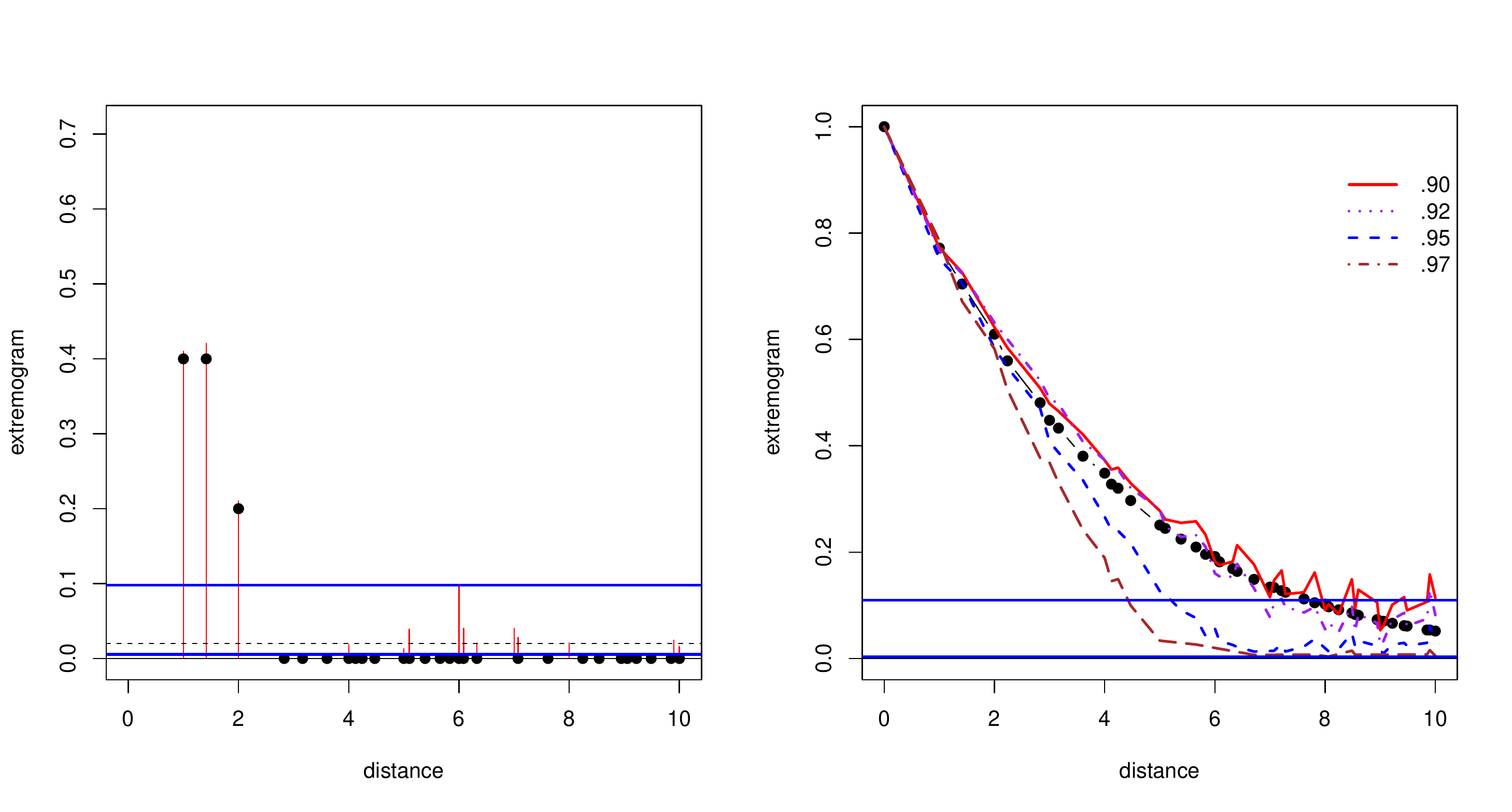}
\caption{$\rho_{AA}(h)$ and $\hat{\rho}_{AA,m}(h)$, where $A=(1, \infty),$ from a realization of an MMA(1) (left) and the process (\ref{def.imma}) (right). For the ESE, $a_m = $ .97  (left) and $a_m = $  (.90,.92,.95,.97) quantile (right) are used. For both cases, the ESE closely tracks the extremogram. Two horizontal lines are 95\% \textit{random permutation} confidence bands.}
 \label{figure:mma1}
 \end{figure}

Now consider $w(s) = \phi^{||s||}$ where $0 < \phi  <1$. Then the process (\ref{def.gen.imma}) becomes
\begin{gather}
\label{def.imma}   {X_t = \max_{s \in \mathbb{Z}^2} \phi^{||s||} Z_{t-s} \quad \text{for} \quad \sum_{l \in \mathbb{Z}^2} \phi^{|| l ||} = \sum_{ 0 \leq ||l|| < \infty} {\phi^{||l||} p(||l||)} < \infty, }
\end{gather}
\YB{where $p(||l||)= \# \{s \in \mathbb{Z}^2: d(\textbf{0}, s) = ||l|| \}$}. Observe that the process (\ref{def.imma}) is istotropic and that $p(||l||) = O(||l||)$ \YB{from Lemma A.1 in \citet{jenish}}, and  
\begin{eqnarray} 
 P(X_t \leq x)  &  = & \exp{ \Big\{  - \frac{1}{x} \sum_{ 0 \leq ||l|| < \infty} {\phi^{||l||} p(||l||)}  \Big\} }, \label{imma.joint.prob.1} \\
 P(X_{ \textbf{0} } \leq x, X_{h} \leq x )  & = &   \exp{ \Big\{  -  \frac{1}{x}  \sum_{s \in \mathbb{Z}^2}  \max{ ( \phi^{||s||}, \phi^{||h+s||})}  \Big\}} \nonumber \\
		&  = & \exp{ \Big\{  -  \frac{1 }{x}  \sum_{ 0 \leq ||l|| < \infty}  \phi^{||l||} q(||l||)   \Big\},}   \label{imma.joint.prob.2}
\end{eqnarray}
\YB{where $ q(||l||)= \# \{s \in \mathbb{Z}^2: \min (||s||, ||h+s|| )= ||l|| \}$, the number of observations with minimum distance to $\textbf{0}$ or $h$ equals $||l||$.  For a  given $h$, if $ ||l|| < \frac{||h||}{2}$, there are $ p(||l||)$ pairs from both $\textbf{0}$ and $h$ while $q(||l||)/ p(||l||) \rightarrow 1 $ as $||l|| \rightarrow \infty$. In other words,}
\begin{center} 
$q(||l||) = 2 p(||l||)$ for $ ||l|| < \frac{||h||}{2}$ and $ \displaystyle \lim_{ ||l|| \rightarrow \infty} \; \frac{q(||l||)}{ p(||l||)} = 1$.
\end{center} 
Using the joint distribution in (\ref{imma.joint.prob.2}) and a Taylor series expansion, the extremogram with $A=B=(1,\infty)$ is
\begin{gather} \label{def.imma.extremo}
 \rho_{(1, \infty)(1, \infty)}(h) =   \frac{  \sum_{\frac{||h||}{2} \leq ||l|| < \infty} {\phi^{||l||} [ 2 p(||l||) - q(||l||)] } }{ \sum_{0 \leq ||l|| <  \infty} {\phi^{||l||} p(||l||)} }.
\end{gather}

\begin{example} \label{example1.imma}
For the process (\ref{def.imma}), the conditions (\ref{cond.2})-(\ref{cond.3}) in Theorem  \ref{cor:extremogramclt} are satisfied if $ r_n^2=o(m_n), \log m_n = o(r_n)$ and $ \log n = o(m_n)$.
\end{example}
\begin{proof}
Observe that (\ref{def.imma}) is isotropic. \YB{By Lemma A.1 in \citet{jenish}}, $p(||l||) = O(||l||).$ Thus, (\ref{dombry}) implies that
\begin{center}
$\alpha_{c,c}(k) \leq const \int_{\frac{k}{2}}^{\infty} j \phi^j dj =  O(k \phi^{k}) $ for any $k>0$. 
\end{center}
Then (\ref{cond.1}) is satisfied  if $\log m_n = o(r_n)$ since 
\begin{center}
$ \displaystyle  m_n \sum_{ l \in \mathbb{Z}^2, r_n < ||l||} \alpha_{c,c }(||l||)  = m_n \sum_{r_n < ||l||} p(||l||) \alpha_{c,c}(||l||)      =  O \left(  m_n   \;   r_n^2 \phi^{r_n} \right).$
\end{center} 
Similarly, (\ref{cond.000}) can be shown. If $\log n = o(m_n)$, (\ref{cond.3}) holds since (\ref{dombry}) implies
\begin{center}
$ \displaystyle n^{d/2} m_n^{1/2} \alpha_{c,c n^d} (m_n)   \leq const \;  n^{3d/2}  m_n^{1/2}    m_n \phi^{m_n}. $ 
\end{center}
{Turning to (\ref{cond.2}), notice from (\ref{imma.joint.prob.1}) and (\ref{imma.joint.prob.2}) that 
\begin{eqnarray*}
 P\left( \max_{\textbf{s} \in B_{\gamma}} |X_{\textbf{s}}| >  {\epsilon a_m} ,    \max_{ \textbf{s}' \in  B_{\gamma} + l }|X_{\textbf{s}'}| >  {\epsilon a_m}  \right)  & \leq  &
 \sum_{ s \in B_{\gamma}} \sum_{s' \in  B_{\gamma} + l }  P\left( X_{\textbf{s}} >  {\epsilon a_m} ,  X_{\textbf{s}'} >  {\epsilon a_m}  \right)\\
& \leq &  \sum_{ s \in  B_{\gamma}} \sum_{s' \in  B_{\gamma} + l }   \left[ \frac{const}{\epsilon a_m}  \; \sum_{\frac{d(s,s') }{2} \leq  j < \infty} {\phi^{j} j }  + O \left(\frac{1}{a_m^2} \right) \right]\\
& \leq & const \; \frac{\phi^{||l||} ||l|| }{\epsilon a_m}   + O \left(\frac{1}{a_m^2} \right).
 \end{eqnarray*}
Hence the term in (\ref{cond.2}) is bounded by  
\begin{eqnarray*}
\limsup_{n \rightarrow \infty} m_n \sum_{l \in \mathbb{Z}^2, k < ||l|| \leq r_n} \left[ const \; \frac{\phi^{||l||} ||l|| }{\epsilon a_m}   + O \left(\frac{1}{a_m^2} \right) \right]  =   \sum_{ k < ||l|| < \infty} const \;  {\phi^{||l||} ||l||^2 }   + \limsup_{n \rightarrow \infty}  O \left(\frac{m_n r_n^2}{a_m^2} \right),
\end{eqnarray*}
where the second term is 0 \YB{since $a_m = O(m_n)$ and $r_n^2=o(m_n)$.} Now letting $k \rightarrow \infty$, we obtain  (\ref{cond.2}).}
\end{proof}

Figure \ref{figure:mma1} (right) shows $\rho_{AB}(h)$ and $\hat{\rho}_{AB,m}(h)$ from a realization of the process (\ref{def.imma}) with $\phi = 0.5$. Here, $A = B = (1, \infty)$ and $a_m = $ (.90,.92,.95,.97) quantiles. The dots are ${\rho}_{AB}(h)$ and the dashed lines are $\hat{\rho}_{AB,m}(h)$ with different $a_m$. The ESE with $a_m = .90$ and $.92$ are close to the extremogram for all observed distances while the ESE with $a_m = .95$ and $.97$ quantiles decay faster for the observed distances greater than 3. The two horizontal lines are 95\% confidence bands based on random permutations.

\subsection{Brown-Resnick process} \label{Example_BR}
{We begin with the definition of the Brown-Resnick process with Fr\'{e}chet marginals. Details can be found in \citet{kabluchko} or \citet{christina}. Consider a stationary Gaussian process $\{ Z_s, s \in \mathbb{R}^d \}$ with mean 0 and variance 1 and use $\{ Z_s^j, s \in \mathbb{R}^d \}, j \in 1,...,n,$ to denote independent replications of $\{ Z_s, s \in \mathbb{R}^d \}$. For the correlation function $\rho(h) = E[Z_s Z_{s+h}]$, assume that there exist sequences $d_n \rightarrow 0$ such that 
\begin{gather*}
\log (n) \{ 1 -  \rho(d_n h) \} \rightarrow \delta(h) >0, \quad \text{as} \quad n \rightarrow \infty. 
\end{gather*}
Then, the random fields defined by
\begin{gather} \label{Def.BR}
 X_s (n)  =  \frac{1}{n} \bigvee_{i = 1}^{n} - \frac{1}{\log \left( \Phi ( Z^i_s) \right) }, \qquad  s \in \mathbb{R}^d, n \in \mathbb{N},
\end{gather}
converge weakly in the space of continuous function to the stationary Brown-Resnick process
\begin{eqnarray} \label{Def.MS}
X_s = \sup_{j \geq 1} \Gamma_j^{-1} Y^j_s = \sup_{j \geq 1} \Gamma_j^{-1}   \exp\{W_s^j - \delta(s)\} , \qquad s \in \mathbb{R}^d,
\end{eqnarray}
where  $(\Gamma_i)_{i \geq 1}$ is an increasing enumeration of a unit rate Poisson process, $\{Y_s^j, s\in \mathbb{R}^d\} , j \in \mathbb{N},$ are iid sequences of random fields independent of $(\Gamma_i)_{i \geq 1}$, and $\{ W_s^j,  s \in \mathbb{R}^d \}, j \in \mathbb{N}, $ are independent replications of a Gaussian random field with stationary increments, $W_{\textbf{0}} = 0 \; \text{and} \; E[W_s] = 0$ and covariance function by $ \text{cov}(W_{s_1}, W_{s_2}) = \delta(s_1) + \delta(s_2) - \delta(s_1-s_2)$. Here, $\Phi$ is the cumulative distribution function of $N(0,1)$. }

\begin{figure} [t!]
\centering
 \includegraphics[width=15cm, height = 7cm]{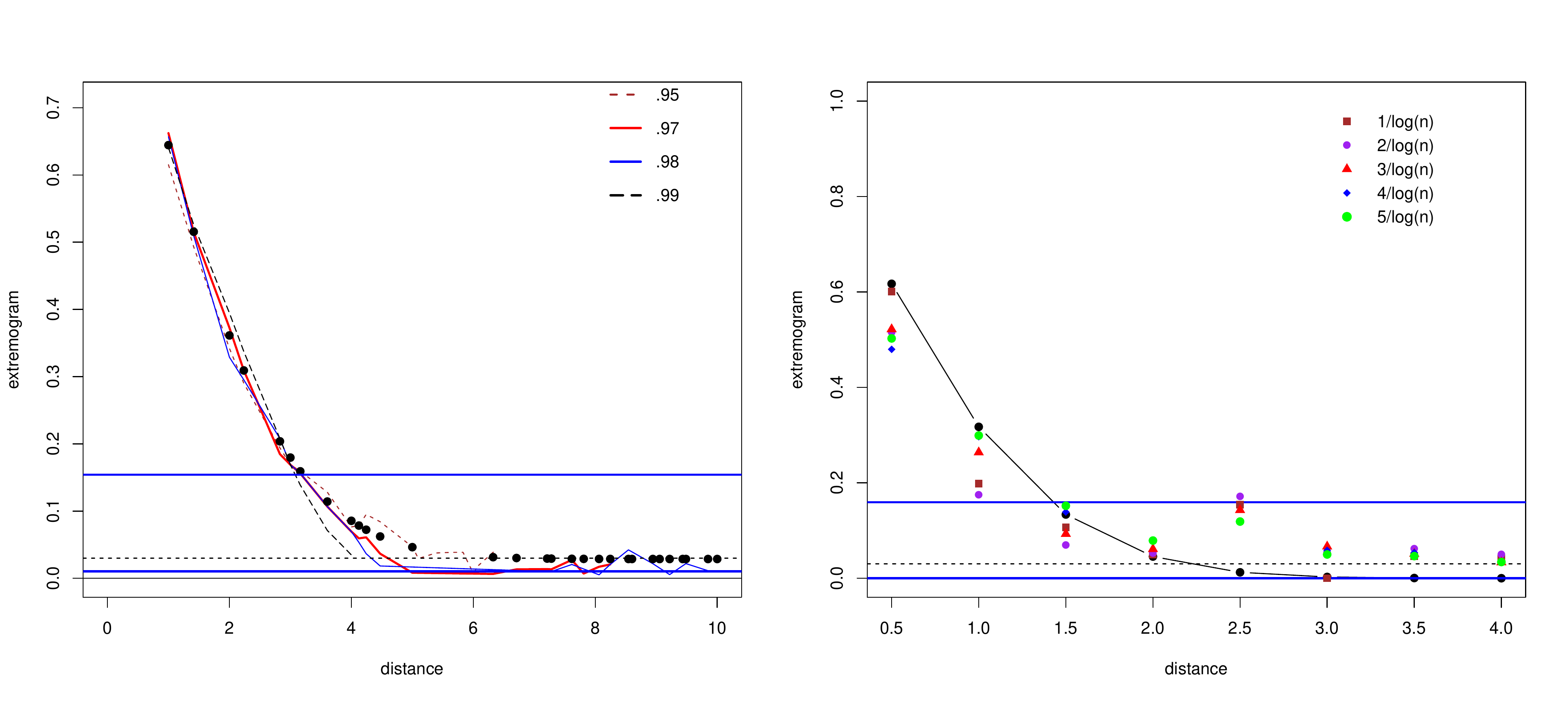}
\caption{$\hat{\rho}_{(1,\infty)(1,\infty),m}(h)$ from a realization of Brown-Resnick process on lattice (left) and non-lattice (right). For lattice case, the ESE with $a_m=(.95,.97,.98,.99)$ upper quantiles are presented. For non-lattice case, the ESE with different bandwidths, $\frac{c}{\log n}$ with $c = 1,2,3,4,$ and $5$, are displayed. Two horizontal lines are 95\% \textit{random permutation} confidence bands.}
\label{figure:smith1_BR1}
 \end{figure}

The extremogram for the Brown-Resnick process $\{ X_s,   s \in \mathbb{R}^d \}$ with $A = (c_A, \infty)$ and $B = (c_B, \infty)$ is 
\begin{gather} \label{extremogram.br}
\rho_{AB}( {h}) =    \bar{\Phi}_{c_A,c_B} \left( \delta(h) \right)  + \frac{c_A}{c_B} \bar{\Phi}_{c_B,c_A} \left( \delta(h) \right),
\end{gather}
where $ \Phi_{y_1,y_2} \left( \delta(h) \right)  = \Phi\left( \frac{\log(y_2/y_1)}{2 \sqrt{\delta(h)}} + \sqrt{\delta(h)} \right) $. To see (\ref{extremogram.br}), recall from \citet{husler} that 
\begin{eqnarray*}
F( y_1, y_2) := P(X_{\textbf{0}} \leq y_1,  X_h \leq y_2) =  \exp \left\{ -\frac{1}{y_1} \Phi\left( \frac{\log(y_2/y_1)}{2 \sqrt{\delta(h)}} + \sqrt{\delta(h)} \right)  -\frac{1}{y_2} \Phi\left( \frac{\log(y_1/y_2)}{2 \sqrt{\delta(h)}} + \sqrt{\delta(h)} \right) \right\}.
\end{eqnarray*}
\YB{As $a_m = O(m_n)$, we assume without loss of generality that $ \lim_{n \rightarrow \infty} \frac{m_n}{a_m} =1$. Then we have $p_m(A) = m_n \left( 1 - e^{ - \frac{1}{c_A a_m} } \right) = \frac{m_n}{c_A a_m} +  O(\frac{m_n}{a_m^2}) \rightarrow \frac{1}{c_A} = \mu(A) $ and}
	\begin{eqnarray}
	\tau_{AB,m}(h) =  m_n \left[ 1 - e^{ - \frac{1}{c_A a_m} }  -  e^{ - \frac{1}{c_B a_m} } + F(a_m c_A, a_m c_B)  \right]  
		 \rightarrow  \frac{1}{c_A} \bar{\Phi}_{c_A,c_B} (\delta(h)) + \frac{1}{c_B} \bar{\Phi}_{c_B,c_A} (\delta(h)),\label{BR.calculation}
	\end{eqnarray}
which proves (\ref{extremogram.br}).

\YB{Similar to Lemma 2  in \citet{christina}, $\alpha$-mixing coefficient of the process is bounded by}
\begin{gather} \label{BR.alpha.bound}
\alpha_{m,n}(||h||) \leq const  \; \sup_{l \geq  ||h||}  \frac{1}{\sqrt{\delta(l)}}  e^{- \delta(l)/2}.
\end{gather}
{In the following examples, the correlation function $\rho(h)$ of a Gaussian process $\{Z_s, s \in \mathbb{R}^d \}$ is assumed to have an expansion around zero as 
\begin{gather}
\rho(h) = 1 - \theta ||h||^{\alpha} + o(||h||^{\alpha}), \quad h \in \mathbb{R}^d,  \label{cond.corr.BR}
\end{gather}
where $\alpha \in (0,2]$ and $\theta >0$. For this choice of correlation function, we have $\delta(h) = \theta ||h||^{\alpha}$ as mentioned in \citet{christina}, Remark 1.

\begin{example} 
\label{example.br.1} Consider the Brown-Resnick process  $\{ X_s , s \in \mathbb{Z}^d\}$ with $\delta(h) =  \theta||h||^{\alpha}$ for $ 0 < \alpha \leq 2$ and $\theta >0$. The conditions of Theorem \ref{cor:extremogramclt} hold if $\log n = o(m_n^{\alpha}), \log m_n = o(r_n^{\alpha})$ and $r_n^d/m_n  \rightarrow 0$. {In this case, (\ref{bias}) is not satisfied for $d>0$}.
\end{example}
\begin{proof} 
From (\ref{BR.alpha.bound}), we have $\alpha_{c,  c} (||h||)   \leq  const \; ||h||^{-\alpha /2}  e^{- \theta||h||^{\alpha} /2}$. If $ \log m_n = o(r_n^{\alpha})$, (\ref{cond.1}) holds since
\begin{eqnarray*}
\displaystyle m_n \sum_{ l \in \mathbb{Z}^d, r_n \leq ||l||}   \alpha_{c,c }(||l||) \leq const \; m_n \sum_{r_n \leq ||l|| < \infty}||l||^{d-1} \alpha_{c,c }(||l||)    \leq  const   \; m_n  \sum_{r_n \leq ||l|| < \infty} ||l||^{d-1-\alpha/2}  e^{- \theta || l||^{\alpha} /2 } \rightarrow 0.  
\end{eqnarray*}
Similarly, (\ref{cond.000}) can be checked. For (\ref{cond.3}), Proposition \ref{dombry.prop} implies that
\begin{center}
$ \displaystyle  n^{d/2} m_n^{1/2} \alpha_{c ,c n^d} (m_n)  \leq const \;  n^{3d/2}  m_n^{(1-\alpha)/2}    \exp\{-  \theta m_n^{\alpha}/2\}$
\end{center}
which converges to 0 if $\log n = o(m_n^{\alpha})$. Showing (\ref{cond.2}) is similar to Example \ref{example1.imma}. From (\ref{BR.calculation}), 
\begin{eqnarray*}
 P\left( \max_{\textbf{s} \in B_{\gamma}} |X_{\textbf{s}}| >  {\epsilon a_m} ,    \max_{ \textbf{s}' \in B_{\gamma} + l }|X_{\textbf{s}'}| >  {\epsilon a_m}  \right)  
& \leq & const \; \frac{ \bar{\Phi}_{(1,\infty), (1,\infty)} (\sqrt{\delta(||l||)} )}{\epsilon a_m}   + O \left(\frac{1}{a_m^2} \right).
 \end{eqnarray*}
Hence the term in (\ref{cond.2}) is bounded by  
\begin{eqnarray*}
& & \limsup_{n \rightarrow \infty} \sum_{ l \in \mathbb{Z}^d, k <||l|| \leq r_n} \left[ const \; m_n \frac{ \bar{\Phi}_{(1,\infty), (1,\infty)} (\sqrt{\delta(||l||)} )}{\epsilon a_m}   + O \left(\frac{1}{a_m^2} \right) \right] \\
 & & \qquad \qquad \qquad \qquad  \qquad \qquad \qquad   \leq    const \;\sum_{k <||l|| < \infty}  ||l||^{d-1}  e^{-  \frac{\theta||l||^{\alpha}}{2}} + \limsup_{n \rightarrow \infty} O \left(\frac{r_n^d m_n}{a_m^2} \right),
\end{eqnarray*}
where the second term is 0 since $r_n^d=o(m_n)$. Letting $k \rightarrow \infty$, (\ref{cond.2}) is obtained. 

\YB{For the last statement in Example \ref{example.br.1}, to show (\ref{bias}) not hold, note that $ \displaystyle| \rho_{AB,m}(h) - \rho_{AB}(h) |  = O \left({1}/{m_n} \right)$ from (\ref{BR.calculation}) and a Taylor series expansion and that $a_m= O(m_n)$ and $\log n = o(m_n^{\alpha})$.}
\end{proof} 

In Figure \ref{figure:smith1_BR1} (left), we have $\rho_{AB,m}(h)$ and $\hat{\rho}_{AB,m}(h)$ from a realization of the Brown-Resnick process with $\delta(h)=  \frac{2}{9} ||h||^2 $. We use 1600 points ($\{1,...,40\}^2 \in \mathbb{Z}^2$) to compute the extremogram with $ A = B = (1, \infty)$ and  $a_m = (.95, .97, .98, .99)$ upper quantiles. The extremogram is marked by dots and the ESE with different line types corresponding to various choices of $a_m$. From the figure, the ESE is not overly sensitive to different $a_m$, but $\hat{\rho}_{(1,\infty)(1,\infty),m}(h)$ with $a_m$ = 0.97 quantile looks most robust. Also the extremal dependence seems to  disappear for $h > 4$ based on the random permutation bands (two horizontal lines).

\begin{example}\label{example.br.2} Consider the Brown-Resnick process $\{X_s, s \in \mathbb{R}^2\}$ with $\delta(h) = \theta ||h||^{\alpha}$ for $\alpha \in (0,2]$ and $\theta>0$. Assume that $\log m_n = o(r_n^{\alpha})$ and
\begin{gather}
\sup_n \frac{ \lambda_n^2  n^{2 a} }{ m_n} < \infty \quad \text{and} \quad  \sup_n \frac{  m_n}{\lambda_n^2 n^{2 a}}< \infty \quad \text{for} \quad 0 < a < 1.
 \label{M000}  
\end{gather}
Then Theorem \ref{theorem:clt:2} applies. Furthermore, (\ref{replace}) holds if $\frac{|S_n| \lambda_n^2}{m_n^3} \rightarrow 0$.  \YB{See Appendix C for the proof.}
\end{example}

\begin{remark}
\YB{Using a similar change of variable technique, as in the proof of Proposition \ref{irregular.numerator.1}, one can verify that condition (\ref{M000}) implies (\ref{clt.cond.2}) with $\delta=1$. We omit the details.} One of the choices that satisfies condition (\ref{M000}) and $\frac{|S_n| \lambda_n^2}{m_n^3} \rightarrow 0$ is $ a=\frac{7}{12},\lambda_n = n^{-1/3}$ and $m_n = n^{1/2}$.
\end{remark}

\YB{To simulate the Brown-Resnick process in $\mathbb{R}^2$, we use \textit{RPbrownresnick} in the \textit{RandomFields} package \footnote{http://cran.r-project.org/web/packages/RandomFields/RandomFields.pdf} in R.} 
Here, we consider $\delta(h)=  0.5 ||h||^2.$  \YB{In each simulation, first we generate 1600 random locations in $\{1,...,40\}^2,$ where the process is simulated with the scale of $\left( {1}/{\log (1600) } \right)^{1/a}$ and $\displaystyle \rho( \cdot )= ( 1 +  c \; ||\cdot||^{a})^{-1}$ with $c = 1$ and $a =2$.} For the ESE computation, we use $A = B = (1, \infty)$, $a_m = $ .97 upper quantile. We set  $w(\cdot) = I_{[- \frac{1}{2},  \frac{1}{2}]}(\cdot)$, and distances $h = (0.5,1,...,4.5,5)$. In Figure \ref{figure:smith1_BR1} (right), the extremogram and ESE from one realization are displayed. The extremogram ${\rho}_{AB}(h)$ corresponds to connected solid circles and $\hat{\rho}_{AB,m}(h)$ for different bandwidths $  \lambda_n$ are displayed in different point types. As will be seen in Section \ref{sim.study}, smaller variances and larger biases are observed for a larger bandwidth. The two horizontal lines are the random permutation bands. 

\subsection{Simulation study} \label{sim.study}
We use a simulation experiment to examine performances of the ESE. Samples are generated from models with Fr\'{e}chet marginals for both lattice and non-lattice cases. For lattice cases, we consider MMA(1) and the Brown-Resnick process with $\delta(h)= 0.5 ||h||^2$. In each simulation, $\hat{\rho}_{AB,m}(h)$  with $A = B = (1, \infty)$ and $a_m = $ .97 upper quantile is calculated for observed distances less than 10. This is repeated 1000 times.

\begin{figure} [t!]
\centering
 \includegraphics[width=14cm, height = 14cm]{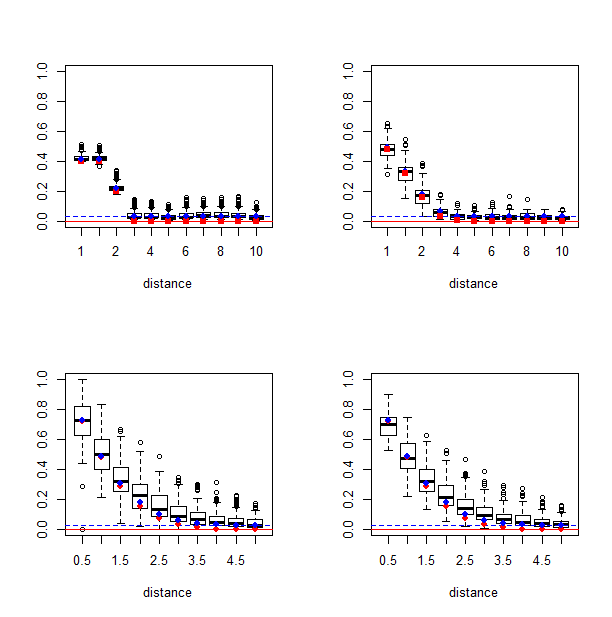}
\caption{The distribution of the ESE for MMA(1) on lattice (upper left, 1000 simulations); the Brown-Resnick process on lattice (upper right, 1000 simulations); on $\mathbb{R}^2$ with $\displaystyle \lambda_n = {1}/{\log{n}}$ (bottom left, 100 simulations); and  $\displaystyle \lambda_n =  {5}/{\log{n}}$ (bottom right, 100 simulations). The solid squares are the extremogram. For MMA(1), we see the ESE is centered around PA extremogram (solid circles). For the Brown-Resnick process on $\mathbb{R}^2$, we see  the impact of bandwidths on the ESE.}
\label{figure:boxplotall}
 \end{figure}

{Figure \ref{figure:boxplotall} (upper left) shows the distributions of $\hat{\rho}_{AB,m}(h)$ (box plots), ${\rho}_{AB}(h)$ (solid squares) and ${\rho}_{AB,m}(h)$ (solid circles) for MMA(1). In the figure, we see the distributions are centered at ${\rho}_{AB,m}(h)$, not ${\rho}_{AB}(h)$. Notice that ${\rho}_{AB,m}(h)$ for MMA(1) is computed by

\begin{eqnarray*}
P(X_h > a_m | X_{\textbf{0}} > a_m) & = & \frac{  1 - 2 P(X_{\textbf{0}}  \leq  a_m) +  P(X_h \leq  a_m , X_{\textbf{0}}  \leq  a_m)  }{P( X_{\textbf{0}} > a_m)}\\
 & = &  \begin{array}{ll} \frac{  \frac{2}{m}-1+ (1 - \frac{1}{m})^{8/5} }{1 / m} & \text{ for } ||h|| = 1, \sqrt{2}, \\
	 \frac{  \frac{2}{m}-1+ (1 - \frac{1}{m})^{9/5} }{1 / m}  & \text{ for }  ||h|| = 2, \\
	 \frac{1}{m}  & \text{ for }  ||h|| > 2 . \end{array}  
\end{eqnarray*}
using $P(X > a_m) = \frac{1}{m}$ and $P(X \leq x) = e^{-5/x}$ for $x>0,$ and $m =  0.03^{-1}.$

{The upper right panel of the figure presents the distributions of the ESE with ${\rho}_{AB}(h)$ (solid squares) and ${\rho}_{AB,m}(h)$ (solid circles)  for the Brown-Resnick process on the lattice. The derivation of $\hat{\rho}_{AB,m}(h)$ is from (\ref{BR.calculation}). Again, the ESE is centered around PA extremogram.  

The bottom panels of Figure \ref{figure:boxplotall} are based on the simulation results from the Brown-Resnick process in the non-lattice case. For each simulation, 1600 points are generated from a Poisson process in $\{1,...,40\}^2$, from which $\hat{\rho}_{AB,m}(h)$ for $h = (0.5,1,...,4.5,5)$ is computed using the bandwidths $  \lambda_n = {1}/{\log{n}}$ and  $ 5/\log{n}$. This is repeated 100 times. Notice that the ESE using $\lambda_n = {1}/{\log{n}}$ has generally smaller bias but larger variance compared to the ESE using $\lambda_n = {5}/{\log{n}}$ for $h \leq 2$. For longer lags, the differences is not apparent. This indicates that the ESE with wider bandwidths tends to have smaller variance but larger biases.} 


\begin{figure} [t!]
\centering
 \includegraphics[width=6cm, height = 5cm]{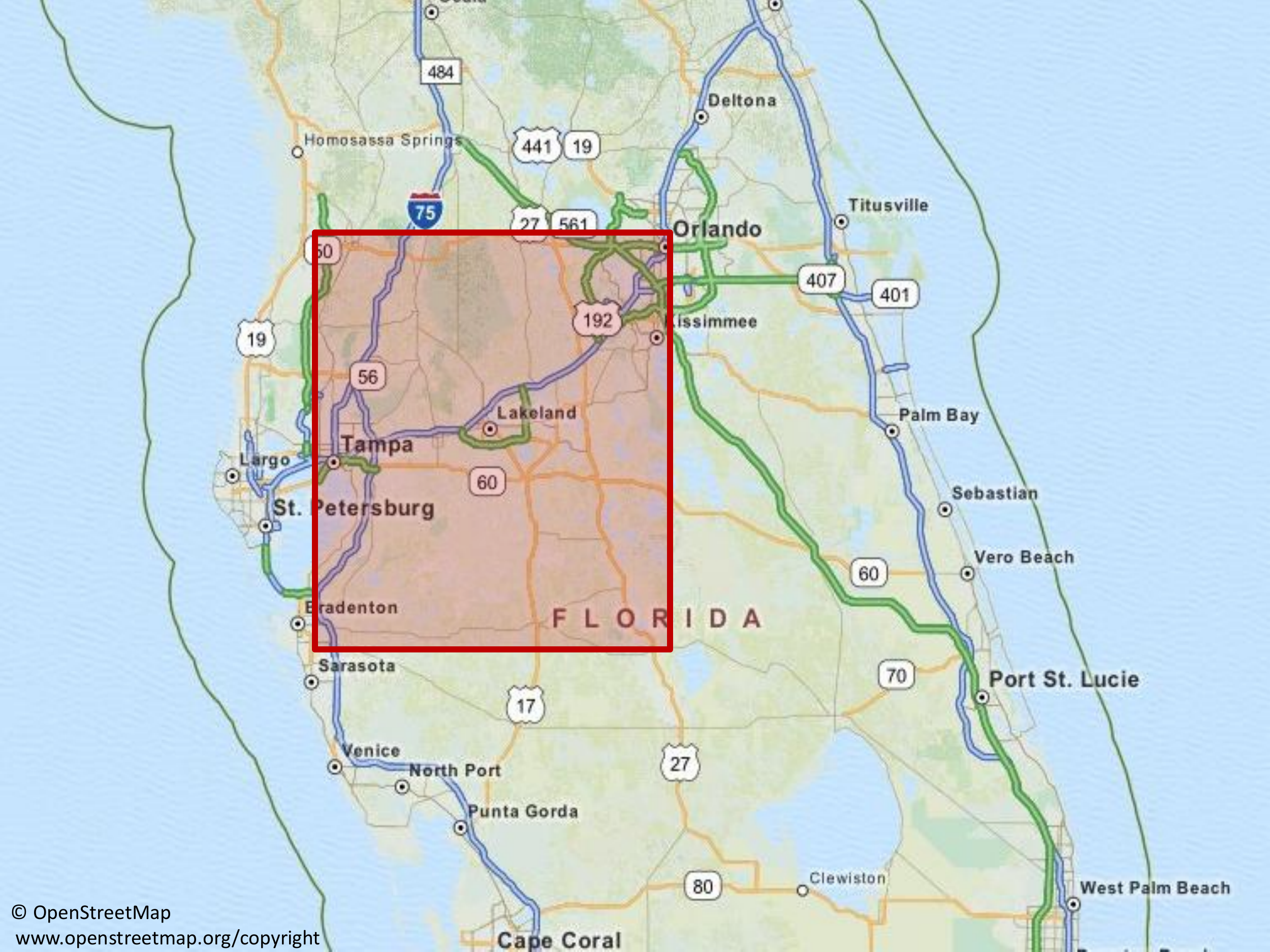}
\caption{The region of Florida rainfall data.}
\label{figure:map}
 \end{figure}

\begin{figure} [t!]
\centering
 \includegraphics[width=15cm, height = 7cm]{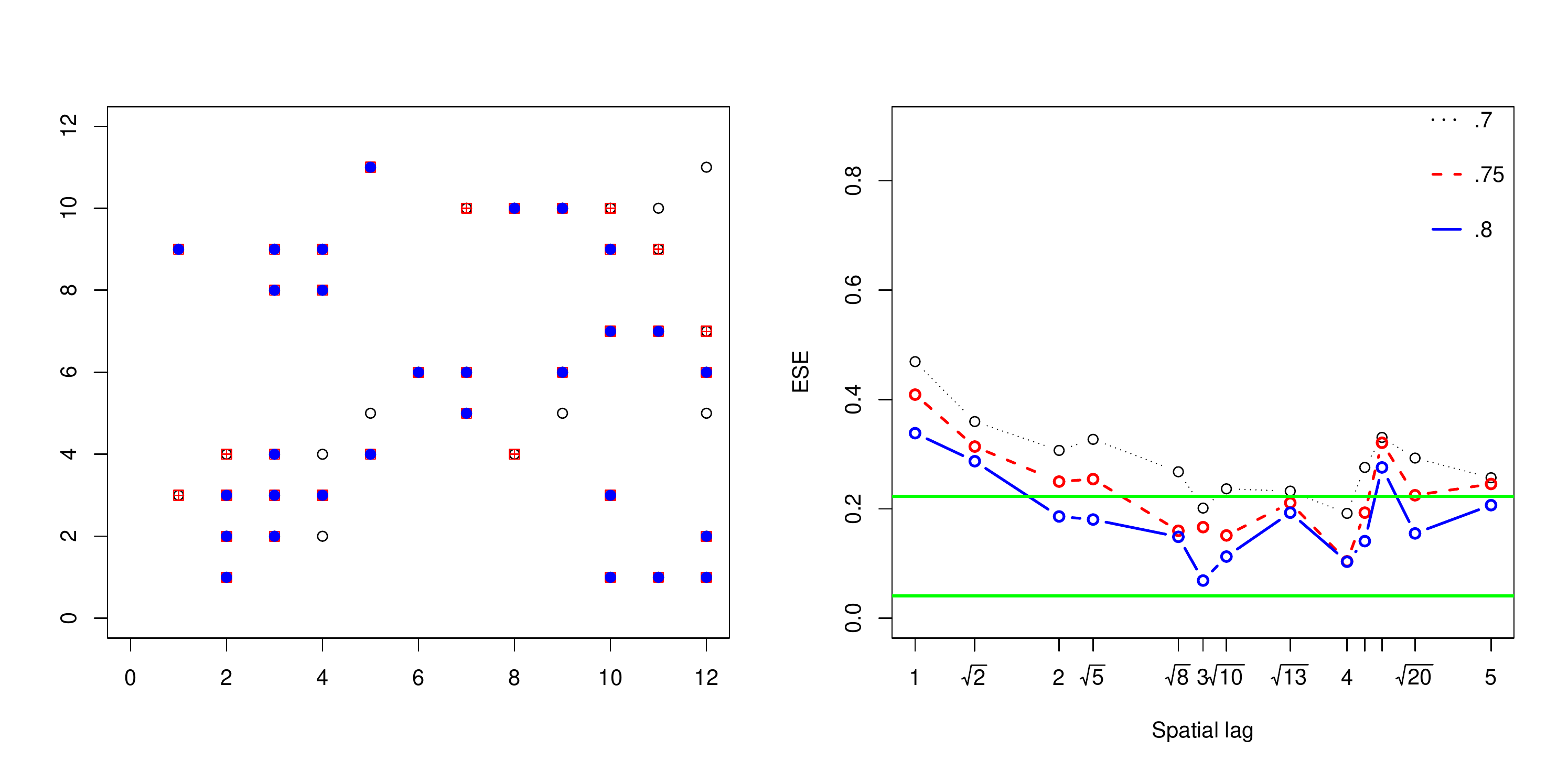}
\caption{The locations of extremes (left) and the ESE (right) using the 6 year maxima of Florida rainfall data.  For example, the ESE with 0.80 upper quantile (solid line, right) is based on the locations of corresponding extremes (solid circles, left). The ESE using the 0.70 upper quantile indicates that no spatial extremal dependence for  lags larger than 3.}
\label{figure:application_rst_1}
 \end{figure}

\section{Application} 
\label{section.application}
In this section, we apply the ESE to analyze geographical dependence of heavy rainfall in a region in Florida. The source is Southwest Florida Water Management District. The raw data is total rainfall in 15 minute intervals from 1999 to 2004, measured on a 120 $\times$ 120 (km)$^2$ region containing 3600 grid locations. The region of the measurements is shown in Figure \ref{figure:map}. For each fixed time, we first calculate the spatial maximum over a non-overlapping block of size 10 $\times$ 10 (km)$^2$, which provides a 12 $\times$ 12 grid of spatial maxima. Then, we calculate the annual maxima from 1999 to 2004 and the 6 year maxima from the corresponding time series for each spatial maximum. The 7 spatial data sets on a 12 $\times$ 12 grid under consideration consist of  annual maxima and 6 year maxima of spatial maxima. Since the data are constructed as a maxima over a spatial grid of 25 locations and a temporal resolution of 15 minutes intervals, it is not unreasonable to view these 7 spatial data sets as realizations from a max-stable process.  

We first look at the spatial extremal dependence for 6 year maxima rainfall. In Figure \ref{figure:application_rst_1}, the locations of extremes (left) and the ESE (right) are displayed, where the ESE is computed using $A=B=(1, \infty)$ and $a_m =$ .70 (dotted line), .75 (dashed line) and .80 (solid line) upper quantiles. Since the number of spatial locations is small (144), we chose modest thresholds in order to ensure enough exceedances for estimation of the ESE. Such thresholds should provide good estimates of the pre-asymptotic extremogram for a max-stable process. The locations of extremes are marked corresponding to choices of $a_m$ by  .70 (empty circles), .75 (empty squares) and .80 (solid circles) upper quantiles. For the ESE plot, the horizontal lines are permutation based confidence bands. For example, if extreme events are defined by any rainfall heavier than the $.70$ upper quantile of the maxima rainfall observed for the entire periods, there is a significant extremal dependence between two clusters at distance 2. On the other hand, using the $0.80$ upper quantile, the extremal dependence at the same distance is no longer significant. In the case of 6 year maxima rainfall, the ESE from the 0.70 upper quantile indicates that no spatial extremal dependence for spatial lags larger than 3. \YB{A small spike of the ESE at spatial lags around 4 may be the result of two extremal clusters that are 4 units apart, as seen in the left panel of Figure \ref{figure:application_rst_1}.}

By looking at the ESE of annual maxima rainfall from 1999 to 2004, we see year-over-year changes in spatial extremal dependence. Figure \ref{figure:application_rst_2} presents the locations of extremes and the ESE from 1999 to 2004 (left to right, top to bottom). For example, the ESE suggests that the spatial extremal dependence for lags less than 3 in 2000 is stronger than at any other year between 1999 and 2004. Using the .80 upper quantile, there is  significant extremal dependence for spatial lag $\sqrt{8}$ in 2000, but not for any other years. In 2002, the spatial extremal dependence is not significant at lag $\sqrt{8}$ using the .80 upper quantile. Similarly, the year-to-year comparisons of the ESE with 0.70 and 0.75 upper quantiles confirm that the spatial extremal dependence for spatial lags up to 3 is stronger in 2000 than in any other years.

\begin{figure} [t!]
\centering
 \includegraphics[width=12cm, height = 16cm]{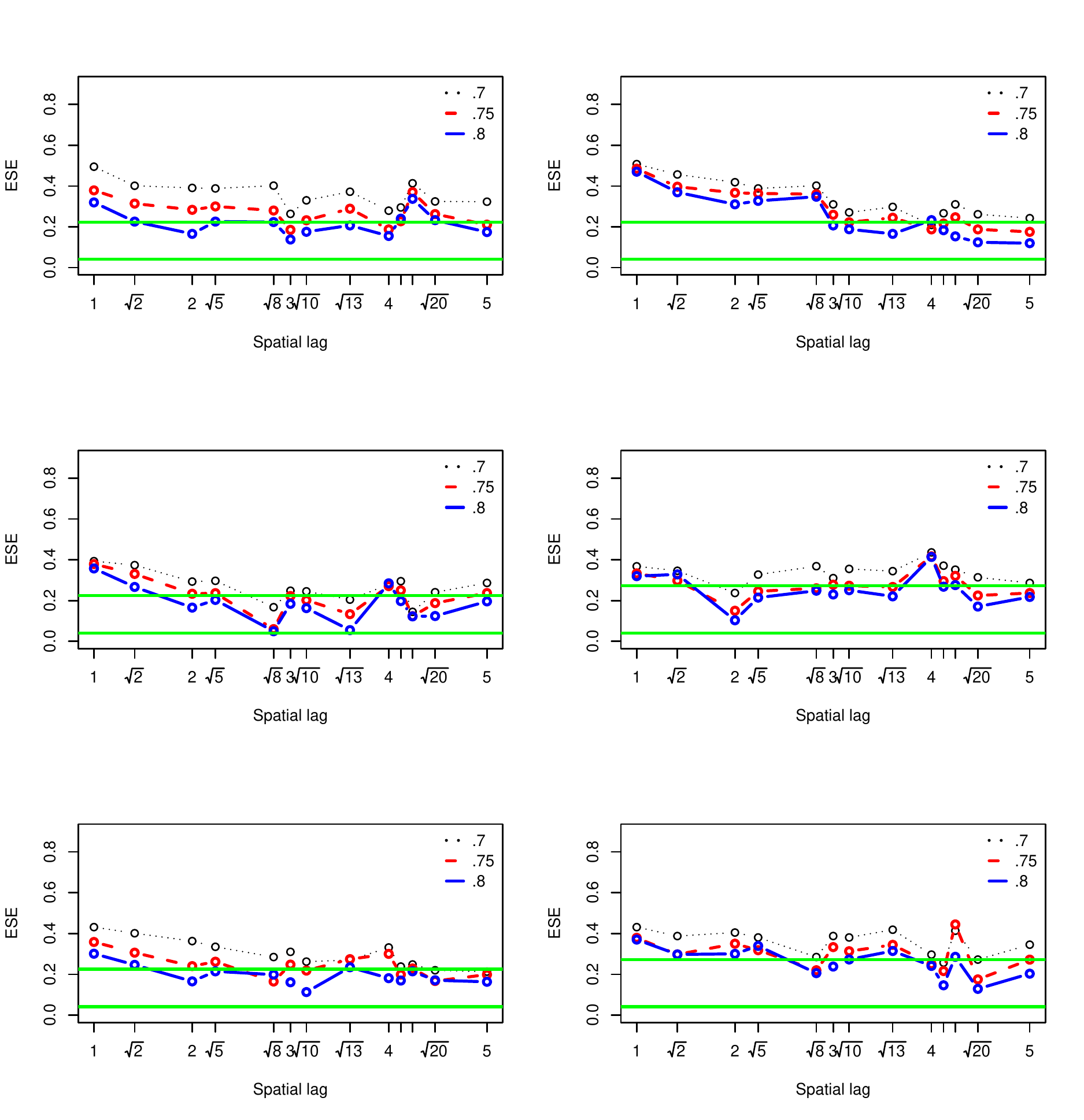}
\caption{The ESE of the annual maxima of Florida rainfall from 1999 to 2004 (left to right, top to bottom). The ESE indicates that the spatial extremal dependency for spatial lags less than 3 is the strongest in 2000.}
\label{figure:application_rst_2}
 \end{figure}

\section{Appendix: Proofs} \label{Section: Appendix}
{The following proposition presented by \citet{lgs} is used in the proof. The proposition is analogous to Theorem 17.2.1 in \citet{ibragimov(1971)}.
\begin{Proposition} [Lemma A.1. in \citet{lgs}] \label{covar.upper}
Let $U$ and $V$ be two closed and connected sets in $ \mathbb{R}^d$ such that $ \# U=\# V \leq b$ and $d(U,V) \geq r$ for some constants $b$ and $r$. \YB{For a stationary process $X_s$, consider $\xi$ and $\eta$  measurable random variables with respect to $ \sigma(X_s: s\in U)$ and $ \sigma(X_s: s \in V)$} with $|\xi| \leq C_1, |\eta| \leq C_2$. Then $ |\text{cov}(\xi, \eta)| \leq 4 C_1 C_2 \alpha_{b,b}(r)$.
\end{Proposition} }

\subsection{Appendix A: Proof of Theorem \ref{cor:extremogramclt}} \label{appendix A}
Theorem \ref{cor:extremogramclt} is derived from Theorem \ref{cor:vec}. For notation, we suppress the dependence of $m$ on $n$ and write $m$ for $m_n$. Define a vector valued random field by
\begin{center}
 $Y_{ {t}}= X_{D_t},$ where $D_t = t + B_{\gamma} = \{  {s} \in \mathbb{Z}^d: d( {t}, {s}) \leq \gamma \}$. 
\end{center} 
In Theorem \ref{cor:vec}, we will establish a joint central limit theorem for
\begin{equation} \label{def.vec.sum}
  \hat{P}_m(C) = \frac{m_n}{n^d} \sum_{{ {t}} \in \Lambda_n} I_ { \{Y_{ {t}} / a_m \in C \} }  =\frac{m_n}{n^d} \sum_{{ {t}} \in \Lambda_n^p } I_ { \{Y_{ {t}} / a_m \in C \} } + \frac{m_n}{n^d} \sum_{ {t} \in  \Lambda_n \setminus \Lambda_n^p} I_ { \{Y_{ {t}} / a_m \in C \} },
\end{equation}
where $\Lambda_n^p = \{t \in \Lambda_n: d(t, \partial \Lambda_n) \geq p \}$ and $ \partial \cdot$ denotes the boundary.
In fact, showing a CLT for the first term in (\ref{def.vec.sum}) is sufficient as the second term is negligible as $n \rightarrow \infty$.
Recall that 
\begin{center}
$  \displaystyle {p}_m(A) =  m P( X_{\textbf{0}} \in a_m A)$  and $\displaystyle{\tau}_{AB,m}(h) = m P(X_{\textbf{0}} \in a_m A,X_h \in a_m B)$,
\end{center}
where $A$ and $B$ are sets bounded away from the origin.  Write  $ \displaystyle \mu(A) = \lim_{n \rightarrow \infty} {p}_m(A) $,
\begin{eqnarray*}
& & \tau_{AB}(h) =  \lim_{n \rightarrow \infty}{\tau}_{AB,m}(h), \\ 
& &   {  \mu_{A}(D_{\textbf{0}}) = \lim_{x \rightarrow \infty}  P \left( \frac{Y_{t}}{\| Y_t \|} \in A \Big|  \| Y_t \|> x \right ), }  \\
& &    {  \tau_{A \times B} ( D_{\textbf{0}} \times D_l) =  \lim_{x \rightarrow \infty}  P \left( \frac{( Y_{{\textbf{0}}}, Y_l)}{\| \text{vector} \{ Y_{\textbf{0}}, Y_l \} \|} \in A \times B \Big|  \| \text{vector}\{ Y_{\textbf{0}}, Y_l \} \|> x \right ).}
\end{eqnarray*}

\begin{theorem} \label{cor:vec}
Assume the conditions of Theorem \ref{cor:extremogramclt}. Let C be a set bounded away from zero and a continuity set with respect to $\mu$ and $\tau$.  Then
  \begin{eqnarray*}
 S_n = \left( \frac{m_n}{n^d} \right)^{1/2}\sum_{{s} \in \Lambda_n} \left[ I \left( \frac{Y_{s}}{a_m} \in C \right) - P\left( \frac{Y_s}{a_m} \in C \right) \right]  \xrightarrow{d} \YB{N}(0, \sigma_Y^2(C) ) 
\end{eqnarray*} 
where   { $       \sigma_Y^2(C)  =   \mu_{C}(D_{\textbf{0}}) +  \sum_{l \neq \textbf{0} \in \mathbb{Z}^d}   \tau_{C \times C} ( D_{\textbf{0}} \times D_l) . $  }
\end{theorem}

\begin{proof}
\YB{We use ideas from \citet{Bolthausen} and \citet{Richard09} to show the CLT for quantity in (\ref{def.vec.sum})
\begin{center}
$\hat{P}_{m}(C) =  {m_{n}} \sum_{ {s} \in \Lambda_n } I _ s/{|\Lambda_n|}$ where $ I_s=  I _ { \left\{ X_{ {s} }/ a_{m} \in C \right\} }$.
\end{center}
The proof for the CLT of  $X_s$ replaced by a vector valued random field $Y_s$ in indicator is analogous.}

Define $H[a, b]  =  \{d(s, t): a \leq d(s,t) \leq b \}$ and {$||l|| = d(\textbf{0},l)$} for convenience. Assume $m_n^{2+2d}=o({n^d})$,  $r_n^{d}=o({m_n}),$ and 
\begin{eqnarray}
&&\label{cond.vec.2} \lim_{k \rightarrow \infty} \displaystyle\limsup_{n \rightarrow \infty} m_{n} \sum_{ l \in \mathbb{Z}^d, ||l|| \in H(k,r_ n]}  P( | X_{ {l}} | > \varepsilon a_{m},  | X_{ {\textbf{0}}} | > \varepsilon a_{m} ) = 0   \quad \text{for}  \quad \forall \epsilon >0, \\
&&\label{cond.vec.1}   \lim_{n \rightarrow \infty} m_{n} \sum_{ l \in \mathbb{Z}^d, || l || \in H(r_n,\infty)}     \alpha_{1,1}( ||l|| ) = 0, \\
&&\label{cond.vec.5}  { \sum_{l \in  \mathbb{Z}^d}   \alpha_{j_1,j_2}( ||l|| ) < \infty \quad \text{for} \quad 2 \leq  j_1+ j_2 \leq 4,}  \\
&&\label{cond.vec.4}   \lim_{n \rightarrow \infty} n^{d/2} m_n^{1/2} \alpha_{1, n^d}(m_n) = 0, 
\end{eqnarray}
which are univariate case analog of conditions (\ref{cond.2}) - (\ref{cond.3}). 

By the same arguments in \citet{Richard09},
\begin{eqnarray}
&&\label{rst.laattice.1}   E\hat{P}_{m}(C) \rightarrow \mu (C)  \\
&&\label{rst.laattice.2}  \text{var} \left(\hat{P}_{m}(C)\right) \sim \displaystyle\frac{m_{n}}{n^d} \left[ \mu (C) +   { \sum_{ l \neq \textbf{0} \in \mathbb{Z}^d } }  \tau_{CC} (l)  \right] = \frac{m_{n}}{n^d} \;  \sigma_X^2(C),
\end{eqnarray} 
 {where (\ref{rst.laattice.1}) is implied by the regularly varying assumption. To see (\ref{rst.laattice.2}), observe that 
\begin{eqnarray} \label{proof.cov.cal}
 \frac{n^d}{m_{n}}  \text{var} \left(\hat{P}_{m}(C)\right) = \frac{m_{n}}{n^d} \sum_{ s \in \Lambda_n} \text{var} (I_s)  + \frac{m_{n}}{n^d} \sum_{ s,t \in \Lambda_n, s \neq t } \text{cov} (I_s, I_t) = A_1 + A_2.
\end{eqnarray} 
By the regularly varying assumption, $A_1 =  p_m(C) + (p_m(C))^2/m_n \rightarrow \mu(C).$ Turning to $A_2 $, for $k \geq 1$ fixed, 
\begin{eqnarray*}
A_2 & \sim &  \frac{m_{n}}{n^d} \sum_{ l = (l_1,...l_d) \neq \textbf{0}, ||l|| \leq  \max \Lambda_n } \Pi_{i=1}^d (n- |l_i|) \text{cov}(I_{\textbf{0}}, I_l)    \\
	& = & \frac{m_{n}}{n^d} \left[ \sum_{ l \in  \mathbb{Z}^d,  ||l|| \in H(0,k]} \cdot +  \sum_{ l \in  \mathbb{Z}^d,  ||l||  \in H(k, r_n]} \cdot  + \sum_{  l \in  \mathbb{Z}^d, ||l||  \in H(r_n, \max \Lambda_n]} \cdot  \right] = A_{21} + A_{22} + A_{23}
\end{eqnarray*} 
where $\max \Lambda_n = \{ \max(d(s,t)): s,t \in \Lambda_n \}$ and $ \Pi_{i=1}^d (n- |l_i|)$ counts a number of cubes with lag $l$ in $\Lambda_n$. 

From the regularly varying assumption, $  \displaystyle \lim_{k \rightarrow \infty} \limsup_{n \rightarrow \infty} A_{21}  =   \sum_{ l \neq {\textbf{0}} \in \mathbb{Z}^d} \tau_{CC}(l)$ since
\begin{gather*}
 \limsup_{n \rightarrow \infty} A_{21} = \sum_{  l \in \mathbb{Z}^d, ||l||   \in H(0,k]} \limsup_{n \rightarrow \infty}   \left( \tau_{CC,m}(C) - p_m(C) \frac{p_m(C)}{m_n} \right) =     \sum_{ l \in \mathbb{Z}^d, ||l||   \in H(0,k]} \tau_{CC}(l).
\end{gather*}
Thus, it is sufficient to show 
\begin{gather*}
 \lim_{k \rightarrow \infty} \limsup_{n \rightarrow \infty} (|A_{22}|+ |A_{23} |) =0
\end{gather*}
to achieve (\ref{rst.laattice.2}).  Recall that $C$ is bounded away from the origin. Notice that  
\begin{eqnarray*}
   A_{22}  & \leq &   const \;    m_{n} \sum_{  l \in \mathbb{Z}^d, ||l||  \in H(k,r_ n] } \left[ P( | X_{l} | > \varepsilon a_{m},  | X_{ \textbf{0}} | > \varepsilon a_{m} )  + \left(\frac{p_m(C)}{m_n}\right)^2 \right],   \\
  A_{23} & \leq &  const \;  m_{n} \sum_{   l \in \mathbb{Z}^d, || l ||  \in H(r_n,\infty)}   \alpha_{1,1}(||l||),
\end{eqnarray*}
so (\ref{rst.laattice.2}) holds assuming (\ref{cond.vec.2}), (\ref{cond.vec.1}) and $r_n^d/m_n \rightarrow 0.$}

Now, we prove 
\begin{gather}
\label{rst.laattice.3}   {  \sqrt{ \frac{n^d}{m_n}} (   \hat{P}_{m}(C) -  p_m(C) )  	= \sqrt{\frac{m_n}{n^d}} \sum_{ {s} \in \Lambda_n}  \bar{I}_s   \xrightarrow{d} N(0, \sigma_X^2(C) ) }
\end{gather} 
 where $  \bar{I}_t = I \left( \frac{X_{t}}{a_m} \in C \right) - P \left( \frac{X}{a_m} \in C \right)$. First, infer from (\ref{proof.cov.cal}) that

\begin{gather}
\label{inequality.proof}  {\frac{m_n}{n^d}} \sum_{s, t \in \Lambda_n}   | \text{cov}(\bar{I}_s, \bar{I}_t)|  < \infty.
\end{gather} 

As the next step, define 
\begin{gather*}
 S_{\alpha,n} =  \sum_{\beta \in \Lambda_n, d(\alpha,\beta) \leq m_n} \sqrt{\frac{m_n}{n^d}}  \bar{I}_{\beta}, \;  v_n =  \sum_{\alpha \in \Lambda_n }  E(  \sqrt{\frac{m_n}{n^d}}  \bar{I}_{\alpha} S_{\alpha,n} ), \; \bar{S}_n = v_n^{-1/2} S_n, \;   \text{and} \;  \bar{S}_{\alpha,n} = v_n^{-1/2} S_{\alpha,n}.
\end{gather*}
From the definition, $ \displaystyle  v_n  \sim {\text{var}( S_{n})} \rightarrow  \sigma^2(C)$.

Now, use Stein's lemma to show (\ref{rst.laattice.3}) as in \citet{Bolthausen} by checking $\lim_{n \rightarrow \infty} E((i \lambda -  \bar{S}_n)e^{i \lambda  \bar{S}_n}) = 0$ for all $ \lambda \in R$.  Write 
\begin{eqnarray*}
(i \lambda -  \bar{S}_n) e^{i \lambda  \bar{S}_n}  & = & i \lambda e^{i \lambda  \bar{S}_n} (1 - v_n^{-1} \sum_{\alpha \in \Lambda_n}  \sqrt{\frac{m_n}{n^d}}  \bar{I}_{\alpha} {S}_{\alpha,n})   - v_n^{-1/2} e^{i \lambda  \bar{S}_n} \sum_{\alpha \in \Lambda_n}   \sqrt{\frac{m_n}{n^d}}  \bar{I}_{\alpha} [1 -e^{  - i \lambda  \bar{S}_{\alpha,n} } - i \lambda  \bar{S}_{\alpha,n} ] \\
&& \qquad \qquad \qquad   \qquad \qquad \qquad  -v_n^{-1/2}   \sum_{\alpha \in \Lambda_n}     \sqrt{\frac{m_n}{n^d}}   \bar{I}_{\alpha}  e^{  - i \lambda  ( \bar{S}_{n}  - \bar{S}_{\alpha,n} )}  \\
& = & B_1 + B_2 + B_3.
\end{eqnarray*}
{We will show $  E|B_1|^2  \rightarrow 0 $. From Proposition \ref{covar.upper}, when $d(\alpha, \alpha') \geq 3 m_n$,
\begin{eqnarray*}
|\text{cov}(\bar{I}_{\alpha}\bar{I}_{\beta},\bar{I}_{\alpha'}\bar{I}_{\beta'})|
 & \leq &   4 \; \alpha_{2,2}(d(\alpha, \alpha') - 2 m_n). 
\end{eqnarray*}
When $ d(\alpha, \alpha') < 3 m_n,$ let $j = \min\{ d(\alpha, \alpha') , d(\alpha, \beta'),d(\beta, \alpha'),d(\beta, \beta')  \}$. Then
\begin{eqnarray*}
|\text{cov}(\bar{I}_{\alpha}\bar{I}_{\beta},\bar{I}_{\alpha'}\bar{I}_{\beta'})|  & \leq & 4 \; \alpha_{p,q}(j)  
\end{eqnarray*}
for $2 \leq p + q \leq 4.$ Given  $ m_n^{2+2d} = o(n^d)$, we have $E|B_1|^2 \rightarrow 0$ since
\begin{eqnarray} \label{inequality.b1}
&& E|B_1|^2 \nonumber \\
 & = &   \lambda^2 v_n^{-2} \sum_{\alpha,\alpha', \beta,\beta',d(\alpha,\beta) \leq m_n,d(\alpha',\beta') \leq m_n} \frac{m_n^2}{n^{2d}} \; \text{cov}(\bar{I}_{\alpha}\bar{I}_{\beta},\bar{I}_{\alpha'}\bar{I}_{\beta'})  \nonumber \\
 & \leq &      \frac{\lambda^2}{ v_n^{2}}  \frac{m_n^2}{n^{2d}} \; \left[  \sum_{ \alpha \in \Lambda_n}  \sum_{ \alpha' \in \Lambda_n \cap   \{ d(\alpha, \alpha') > 3 m_n \} } \sum_{ \beta, \beta' }    \left| \text{cov}(\bar{I}_{\alpha}\bar{I}_{\beta},\bar{I}_{\alpha'}\bar{I}_{\beta'}) \right|     +   \sum_{ \alpha \in \Lambda_n}  \sum_{ \alpha' \in \Lambda_n \cap   \{ d(\alpha, \alpha') \leq 3 m_n \} } \sum_{ \beta, \beta' }      \left| \text{cov}(\bar{I}_{\alpha}\bar{I}_{\beta},\bar{I}_{\alpha'}\bar{I}_{\beta'}) \right|  \right]   \nonumber \\
  &  \leq &   \frac{\lambda^2}{ v_n^{2}}   \frac{m_n^2}{n^{2d}}    \; 4 \; \left[  \sum_{ \alpha \in \Lambda_n}  \sum_{ \alpha' \in \Lambda_n \cap   \{ d(\alpha, \alpha') > 3 m_n \} } \sum_{ \beta, \beta' }    \alpha_{2,2}(d(\alpha, \alpha')  - 2 m_n)    +   \sum_{ \alpha \in \Lambda_n}  \sum_{ \alpha' \in \Lambda_n \cap   \{ d(\alpha, \alpha') \leq 3 m_n \} } \sum_{ \beta, \beta' }     \alpha_{p,q}( j  )   \right] \;   \nonumber \\
  &  \leq &    \frac{ const \lambda^2 m_n^2}{ v_n^{2} n^{2d}}  n^d m_n^{2d} \; \left[ \sum_ { l \in \mathbb{Z}^d,||l|| \in H(3m_n, \infty)}  \alpha_{2,2}(||l||- 2 m_n)    +  \; \sum_{l \in \mathbb{Z}^d, ||l|| \in H[0, 3m_n]}   \alpha_{p,q}(||l||)   \right] \;  \\
& = &  O(m_n^{2 + 2d}/ n^d ). \nonumber
\end{eqnarray}
Notice that in (\ref{inequality.b1}), $ n^d m_n^{2d}$ is from summing over $\alpha$ (giving $n^d$), $\beta$ (giving $ O( m_n^d)$), and $ \beta'$ (giving $O( m_n^d)$) for the first summation. Similarly, for the second summation, $ n^d m_n^{2d}$ is from summing over $ \alpha,\beta$ and  $ \alpha' \; \text{or}\; \beta'$ depending on the location of points. The last equation is from  (\ref{cond.vec.5}). }

Now we show $E|B_2| \rightarrow 0$ provided $ m_n^{2+2d} = o(n^d)$. Recall that $ | e^{ix} -1 - ix | \leq \frac{1}{2} x^2 $. Then
\begin{eqnarray*}
E|B_2|  & \leq &  c v_n^{-1/2} n^d  \sqrt{\frac{m_n}{n^d}}    E \bar{S}_{\alpha,n}^2\\
& =  & c  v_n^{-1/2}   \sqrt{\frac{m_n}{n^d}}  m_n   \sum_{\beta,\beta',d(\textbf{0},\beta) \leq m_n,d( \textbf{0},\beta') \leq m_n}     { E( \bar{I}_{\beta} \bar{I}_{\beta'} )}  \\
& \leq & c    \sqrt{\frac{m_n}{n^d}} m_n^{d+1}   \sum_{ l \in \Lambda_n}    {    E( \bar{I}_{{\textbf{0}}} \bar{I}_{l} )}   \\
& = & O \left( \sqrt{ \frac{m_n^{1 + 2d}}{n^d}} \right)
\end{eqnarray*}
where $ m_n    \sum_{ l \in \Lambda_n}    {    E( \bar{I}_{{\textbf{0}}} \bar{I}_{l} )}   < \infty$ is inferred from (\ref{inequality.proof}).  

Lastly, the condition (\ref{cond.vec.4}) implies $ |E B_3|   \rightarrow 0 $ since
\begin{eqnarray*}
 |E B_3| & \leq & c v_n^{-1/2} n^d\sqrt{\frac{m_n}{n^d}} \alpha_{1, n^d} (m_n) = c n^{d/2} m_n^{1/2}\alpha_{1, n^d} (m_n).
\end{eqnarray*}
Thus, Stein's lemma is satisfied, which completes the proof.
\end{proof} 

\begin{remark}
$\hat{P}_m(C)$ is a consistent estimator of $\mu(C).$ If $\mu(C) = 0,$ \text{var} $(\hat{P}_m(C)) = o(m_n/n^d)$. 
\end{remark}

\begin{remark} \label{Appendix.A.remark}
The conditions (\ref{cond.2}) - (\ref{cond.3}) are derived from (\ref{cond.vec.2}) - (\ref{cond.vec.4}) by replacing univariate process ($X_{ {t}}$)   by vectorized process ($Y_{ {t}}$). In order to see (\ref{cond.2}) is derived from (\ref{cond.vec.2}), for example, consider Euclidean norm for ($Y_{ {t}}$) process. Then, the vectorized analog of (\ref{cond.vec.2}) is
\begin{center}
$ \displaystyle    \lim_{k \rightarrow \infty} \limsup_{n \rightarrow \infty} m_n \sum_{ l \in \mathbb{Z}^d,  || l || \in H(k,r_n]}    P( ||Y_{ {\textbf{0}}}|| > \epsilon a_m,  ||Y_{ {l}}|| > \epsilon a_m ) = 0$,
\end{center}
which holds under (\ref{cond.2}) by triangular inequality, i.e.,
\begin{eqnarray*}
\displaystyle P( ||Y_{\textbf{0}}|| > \epsilon a_m,  ||Y_{ {l}}|| > \epsilon a_m ) & \leq &   P(  \sum_{ {s} \in D_{\textbf{0}}}|X_{ {s}}| > \epsilon a_m,  \sum_{ {s}' \in D_l} |X_ {s'}| > \epsilon a_m )   \leq  P \left( \max_{ {s} \in D_{\textbf{0}}}|X_{ {s}}| >  \frac{\epsilon a_m}{|D_{\textbf{0}}|} , \max_{ {s}' \in D_l}|X_{ {s}'}|  >  \frac{\epsilon a_m}{|D_l|} \right).
\end{eqnarray*}
The rest of the derivations are straightforward.
\end{remark}

\begin{proof} [Proof of Theorem \ref{cor:extremogramclt}]
Apply the Cram\'{e}r-Wold device to Theorem \ref{cor:vec} to achieve the multivariate central limit theorem, then use $\delta$-method to obtain the central limit theorem for the ESE. To specify the limiting variance $\Sigma$, redefine
\begin{eqnarray*}
& &  \mu(A) = \lim_{x \rightarrow \infty}  P \left( \frac{X_{t}}{\| Y_t \|} \in A \Big|  \| Y_t \|> x \right ).
\end{eqnarray*}
Then, $\Sigma = \mu(A)^{-4}  {F} \Pi  {F}^t $ where
\begin{eqnarray*}
\Pi_{i,i}  & = & \mu_{S_i}(D_{\textbf{0}}) + \sum_{ l \neq {\textbf{0}} \in \mathbb{Z}^d}  \tau_{S_i \times S_i} ( D_{\textbf{0}} \times D_l ) \\
\Pi_{i,j}  & = & \mu_{S_i \cap S_j}( D_{\textbf{0}} ) + \sum_{ l \neq {\textbf{0}} \in \mathbb{Z}^d} \tau_{S_i \times S_j} ( D_{\textbf{0}} \times  D_l ) \\
F & = & \begin{pmatrix}
  \mu(S_{ (\# H)+1}) & 0  & 0 & ... & 0 &  -\mu_{S_1}(  D_{\textbf{0}} ) \\
  0 & \mu(S_{(\# H)+1})  & 0 & ... & 0 & -\mu_{S_2}(  D_{\textbf{0}}  )  \\
. & . & . & ... & . &  .\\
. & . & . & ... & . &  .\\
0 & 0  & 0 & ... & \mu(S_{(\# H)+1}) &  -\mu_{S_{(\# H)}}(  D_{\textbf{0}} ) 
 \end{pmatrix}
\end{eqnarray*}
where the sets $S_i$ are chosen such that $ \{ Y_t \in S_i \} = \{  X_t \in A,  X_s \in B: d(t,s) = h_ i \}$ for $h_i \in H$ and $i = 1,..., (\# H)$ and $\{ Y_t \in S_{(\# H)+1} \} = \{ X_t \in A\}$. For more details, see \citet{Richard09}.  \end{proof}


\subsection{Appendix B: Proof of Theorem \ref{theorem:clt:2}} \label{appendix B}

Theorem \ref{theorem:clt:2} is derived from Proposition \ref{irregular.denominator} - \ref{irregular.numerator.2}. Before proceeding to Proposition \ref{irregular.denominator}, we present the following result regarding LUNC. 
\begin{Proposition}
\label{local.uniform.rv} Consider a strictly stationary regularly varying random field $\{X_s, s \in \mathbb{R}^d \}$ with index $\alpha > 0$ satisfying LUNC. For a positive integer $k$ and $\lambda_n \rightarrow 0$,
\begin{eqnarray*}
	n P\left( \frac{X_{\textbf{0}}}{a_n} \in A_0, \frac{X_{s_1+\lambda_n}}{a_n}  \in A_1, \cdots, \frac{X_{s_k+\lambda_n}}{a_n}  \in A_k  \right)  \rightarrow \tau_{A_0,A_1,\cdots,A_k}(s_1,\cdots,s_k)
\end{eqnarray*}
provided $A_0 \times A_1 \times \cdots \times A_k$ is a continuity set of the limit measure 
\begin{eqnarray*}
\tau_{A_0,A_1,\cdots,A_k}(s_1,\cdots,s_k) = \lim_{n \rightarrow \infty} n P\left(  {X_{\textbf{0}}}/{a_n} \in A_0,  {X_{s_1}}/{a_n}  \in A_1, \cdots,  {X_{s_k}}/{a_n}  \in A_k  \right).
\end{eqnarray*}
\end{Proposition}
\begin{proof}
 Let $f$ be a continuous function with compact support on $\bar{\mathbb{R}}^{k+1}  \setminus \{ {\textbf{0}} \}.$ Since $f$ has compact support, it is uniformly continuous and hence for every $ \epsilon >0$ there exists $\delta$ such that 
$|f(x_1,x_2, \cdots ,x_{k+1}) - f(y_1,y_2, \cdots, y_{k+1})| < \epsilon   $ whenever $ | (x_1,x_2,\cdots ,x_{k+1})  - (y_1,y_2,\cdots, y_{k+1})| < \delta.$

Let $\tilde{X}_n = (X_{\textbf{0}}, X_{s_1+\lambda_n}, \cdots, X_{s_k+\lambda_n})$ and $\tilde{X} = (X_{\textbf{0}}, X_{s_1}, \cdots, X_{s_k})$. Notice that 
\begin{eqnarray*}
	n E  \left| f\left( \frac{\tilde{X}_n }{a_n}  \right) - f\left( \frac{\tilde{X} }{a_n}  \right) \right| & = & n E| \cdot| I_{\{ \frac{|\tilde{X}_n  - \tilde{X}|}{a_n} > \delta \}} +n E| \cdot| I_{\{ \frac{|\tilde{X}_n  - \tilde{X}|}{a_n} \leq \delta \}} = A_1 + A_2.
\end{eqnarray*}
Let $M = \max  f\left( \frac{\textbf{X}}{a_n}  \right).$ By (\ref{note.2}), there exists $\epsilon  > 0$ such that
\begin{eqnarray*}
	\limsup_n A_1 & \leq & \limsup_n 2  M n \left[ P\left( |X_{s_1+\lambda_n} - X_{s_1} | > \frac{\delta a_n}{k} \right)   + \cdots + P\left(|X_{s_k+\lambda_n}- X_{s_k} | > \frac{\delta a_n}{k} \right) \right]  <  2 M \epsilon
\end{eqnarray*}
since $ |X_{\lambda_n}- X_{\textbf{0}}|    \leq	\sup_{|s| < \delta'} |X_s - X_{\textbf{0}}|$ as $n \rightarrow \infty$ for $|\lambda_n| < \delta'$.
For $A_2$, since the support of $f \in \{ |\tilde{X}| > C \} \subset \{  | X_{\textbf{0}}|  > \frac{C}{k+1} \} \cup   \cdots \cup \{  | X_{s_k} |  > \frac{C}{k+1} \} $ 
\begin{eqnarray*}
	\limsup_n A_2 & \leq & \limsup_n  \epsilon n \left[ P\left( \frac{|\tilde{X}_n|}{a_n} > C \right) + P\left( \frac{|\tilde{X}|}{a_n} > C \right)  \right] \\
		& = & \limsup_n  \epsilon \; n \;2 (k+1) \; P\left( {|X_{\textbf{0}}|}> a_n C/(k+1) \right)  \\
		& = &  \epsilon \; 2 ({k+1}) \tau_{BB}(\textbf{0}), \qquad \text{where} \; B = \{x: x > C/(k+1) \}.
\end{eqnarray*}
Take $\epsilon$ small by choosing appropriate $ \delta$ and $ \delta'$, then for a positive integer $k$ and $\lambda_n \rightarrow 0$,
\begin{gather*}
 n E f\left( \frac{X_{\textbf{0}}, X_{s_1+\lambda_n}, \cdots ,X_{s_k+\lambda_n}}{a_n}  \right) \rightarrow \int f(u_1,u_2,\cdots,u_k) \mu(du_1,du_2,\cdots,du_k)
\end{gather*}
for any continuous function with compact support $f$. Using Portmanteau theorem for vague convergence, we complete the proof. See Theorem 3.2 in \citet{Resnick(HT)}. \end{proof}

We discuss asymptotics of the denominator and the numerator of the ESE in turn. 
\begin{Proposition} \label{irregular.denominator}
Under the setting of Theorem \ref{theorem:clt:2} and  condition (\textbf{M2}),
\begin{center}
	$ \displaystyle E ( \hat{p}_m(A) ) = p_m(A) \rightarrow \mu(A) \quad $  and  $ \quad \displaystyle \frac{|S_n|}{m_n}  {\text{var}} ( \hat{p}_m(A) )  \rightarrow \frac{\mu(A)}{\nu} + \int_{\mathbb{R}^2} \tau_{AA}(y) dy$.
\end{center}
Hence, $ \hat{p}_m(A) \xrightarrow{p} \mu(A)$.
\end{Proposition}

\begin{proof}
By the regularly varying property, $E ( \hat{p}_m(A) ) =  p_m(A) \rightarrow \mu(A)$. 

For $ \text{var} (\hat{p}_m(A))$, recall that $N^{(2)}(ds_1, ds_2) = N(ds_1) N(ds_2) I(s_1 \neq s_2)$ and observe that
\begin{eqnarray*}
& &  E( \hat{p}_m(A)^2) \\
 &  & \quad =  \left(\frac{m_n}{\nu |S_n|} \right)^2 E \left[ \int_{S_n}  I\left(\frac{X_{s_1}}{a_m} \in A \right) N(d s_1) +  \int_{S_n}\int_{S_n}  I\left(\frac{X_{s_1}}{a_m} \in A ,\frac{X_{s_2}}{a_m} \in A \right) N^{(2)}(d s_1,d s_2) \right] \\
 &  & \quad =   \left(\frac{m_n}{\nu |S_n|} \right)^2  \left[ \int_{S_n}  \frac{{p}_{m}(A)}{m_n}  \nu d s_1   +  \int_{S_n}\int_{S_n}  \left[ P\left(\frac{X_{s_1}}{a_m} \in A ,\frac{X_{s_2}}{a_m} \in A \right) - \frac{{p}_{m}(A)^2}{m_n^2}  \right] \nu^2 d s_1 d s_2 \right]  +  E ( \hat{p}_m(A) )^2 \\
 &  & \quad =  \left(\frac{m_n}{ |S_n|} \right)  \left[ \frac{ E( \hat{p}_m(A))}{\nu} +  \int_{S_n - S_n} m_n \left[\frac{{\tau}_{AA,m}(y)}{m_n} - \frac{{p}_{m}(A)^2}{m_n^2}  \right]   \frac{|S_n \cap (S_n - y)|}{|S_n|} dy  \right]  +  E ( \hat{p}_m(A) )^2
\end{eqnarray*}
where the change of variables $s_2-s_1=y$ is used in the last line. Using the above, we show
\begin{eqnarray}
\frac{|S_n|}{m_n}  \text{var} ( \hat{p}_m(A) )  & = &   \frac{E ( \hat{p}_m(A) ) }{\nu} + \int_{S_n - S_n} m_n \left[ \frac{{\tau}_{AA,m}(y)}{m_n} - \frac{{p}_{m}(A)^2}{m_n^2}  \right]   \frac{|S_n \cap (S_n - y)|}{|S_n|} dy \nonumber \\ 
	\label{integration}	& \rightarrow & \frac{\mu(A)}{\nu} + \int_{\mathbb{R}^2} \tau_{AA}(y) dy. 
\end{eqnarray}
To see (\ref{integration}), notice that for a fixed $k>0$
\begin{eqnarray*}
\int_{S_n - S_n} m_n \left[ \frac{{\tau}_{AA,m}(y)}{m_n} - \frac{{p}_{m}(A)^2}{m_n^2}  \right]   \frac{|S_n \cap (S_n - y)|}{|S_n|} dy  & = & \int_{B[0,k)} [\cdot] dy +  \int_{B[k,r_n]} [\cdot] dy +  \int_{ (S_n - S_n) \setminus B[0,r_n]} [\cdot] dy  \\
 & = & A_1 + A_2 + A_3.
\end{eqnarray*}
For each fixed $k >0$, $ \displaystyle \lim_{n \rightarrow \infty} A_1 = \int_{B[0,k)} \tau_{AA}(y) dy $. Now, we show  
\begin{center}
$ \displaystyle \lim_{k \rightarrow \infty}\limsup_{n \rightarrow \infty}( |A_2 + A_3|) = 0$. 
\end{center}
Recall that $A$ is bounded away from the origin. Using (\ref{M2}) and $r_n^2=o(m_n)$,
	\begin{eqnarray*}
 |A_2 | & \leq &  \int_{B[k,r_n]} m_n P(|X_y| > \epsilon a_m , |X_{\textbf{0}}| > \epsilon a_m) dy+ const \;  r_n^2  \frac{{p}_{m}(A)^2}{m_n}  \rightarrow 0
	\end{eqnarray*}
From (\ref{M1}), $ \displaystyle \lim_n |A_3| \leq  \lim_n \int_{\mathbb{R}^2 \setminus B[0,r_n)} m_n \alpha_{1,1}(y) dy = 0$. This completes the proof. \end{proof}

\begin{Proposition} \label{irregular.numerator.1} 
Assume that a stationary regularly varying random field satisfies LUNC. Further, assume the conditions of Proposition \ref{irregular.denominator}, and (\ref{irregular.numerator.integration}) in (\textbf{M3}). Then 
\begin{flushleft}
	(i) $ \displaystyle E \hat{\tau}_{AB,m}(h)  \rightarrow \tau_{AB}(h),$ \\
	\ \\
	(ii)   	$  \displaystyle  \frac{|S_n|\lambda_n^2}{m_n} \; \textnormal{\text{cov}}     \left(\hat{\tau}_{AB,m}(h_1),  \hat{\tau}_{AB,m}(h_2) \right) \rightarrow  \frac{\int_{\mathbb{R}^2} w(y)^2 dy }{\nu^2 }   \left[ \tau_{ AB}(h_1) \; I_{\{ h_1 = h_2 \}} +  \tau_{ A \cap B A \cap B}(h_1) \; I_{\{ h_1 = -  h_2 \}}  \right]$, and\\
	\ \\
	(iii) $ \displaystyle   \frac{|S_n|\lambda_n^2}{m_n}  \;   \textnormal{\text{var}}  \left(\hat{\tau}_{AB,m}(h)  \right) \rightarrow  \frac{1}{\nu^2 }  \left( \int_{\mathbb{R}^2} w(y)^2 dy \right)  \tau_{ AB}(h).$
\end{flushleft}
\end{Proposition}

\begin{proof} 
(i) From (\ref{def.num.est.RF}) and stationarity of $\{X_s, s\in \mathbb{R}^2 \}$ 
\begin{center}
	$ \displaystyle E \hat{\tau}_{AB,m}(h) =  \frac{m_n}{\nu^2} \frac{1}{|S_n|} \int_{S_n} \int_{S_n} w_n(h + s_1 - s_2)  \;  P \left(\frac{X_{\textbf{0}}}{a_m} \in A, \frac{X_{s_2-s_1}}{a_m} \in B \right)   \nu^2 ds_1  ds_2 $
\end{center}
which after making the transformation $  \frac{h + s_1 - s_2}{\lambda_n} = y$ and $ s_2 = u$ becomes
\begin{eqnarray*}
 & & \frac{1}{|S_n|} \int_{ \frac{S_n - S_n +h}{\lambda_n }} \int_{S_n \cap (S_n - \lambda_n y + h)}  w({y})  \; {\tau}_{AB,m}(h - y \lambda_n)  du  dy  \\ 
	& = &    \int_{\frac{S_n - S_n +h}{\lambda_n }}   w({y})  \; {\tau}_{AB,m}(h - y \lambda_n)  
\frac{|S_n \cap (S_n - \lambda_n y + h)|}{|S_n|}  dy   \\
	& \rightarrow &    \tau_{AB}(h).
\end{eqnarray*}
The limit in the last line follows from LUNC and the dominated convergence theorem  since
\begin{center}
 $  \displaystyle  \; {\tau}_{AB,m}(h - y \lambda_n)  \frac{|S_n \cap (S_n - \lambda_n y + h)|}{|S_n|}   \leq   \; {p}_{m}(A)     $ and $  \displaystyle \int_{\mathbb{R}^2} w({y})  \; {p}_{m}(A)  dy < \infty $. 
\end{center}
\ \\
(ii)
For fixed sets $A$ and $B$ let $  \tau_m^*(s_1,s_2,s_3,s_4) = m_n  P \left(  \frac{X_{s_1}}{a_m} \in A,  \frac{X_{s_2}}{a_m} \in B , \frac{X_{s_3}}{a_m} \in A,  \frac{X_{s_4}}{a_m} \in B   \right)$. Then, 
\begin{eqnarray}
\label{trash1} & &  \frac{|S_n|\lambda_n^2}{m_n}  E \left(\hat{\tau}_{AB,m}(h_1) \hat{\tau}_{AB,m}(h_2) \right)  \\
&& = \frac{m_n \lambda_n^2 }{\nu^4 |S_n|} \iiiint\limits_{S_n^4}  w_n(h_1 + s_1 - s_2)   w_n(h_2 + s_3 - s_4)  \;   \frac{\tau_m^*(s_1,s_2,s_3,s_4) }{m_n}  \; E [ N^{(2)}(ds_1, ds_2) N^{(2)}(ds_3, ds_4) ] \nonumber 
\end{eqnarray}
where $N^{(2)}(ds_1, ds_2) = N(ds_1) N(ds_2) I(s_1 \neq s_2)$ and 
\begin{eqnarray} \label{karr.equation}
&& E [N^{(2)}(ds_1, ds_2) N^{(2)}(ds_3, ds_4)]   = \nu^4 ds_1 ds_2 ds_3 ds_4 + \nu^3 ds_1 ds_2 \varepsilon_{s_1}(ds_3) ds_4  + \nu^3 ds_1 ds_2 \varepsilon_{s_2}(ds_3) ds_4  \nonumber   \\
 & &   + \nu^3 ds_1 ds_2 ds_3 \varepsilon_{s_1}(ds_4)  + \nu^3 ds_1 ds_2 ds_3 \varepsilon_{s_2}(ds_4) + \nu^2 ds_1 ds_2  \varepsilon_{s_1}(ds_3) \varepsilon_{s_2}(ds_4)  + \nu^2 ds_1 ds_2   \varepsilon_{s_1}(ds_4) \varepsilon_{s_2}(ds_3) \nonumber \\
\end{eqnarray} 
(see  \citet{karr}). Now, let $I_i \; \text{for} \; i = 1,..., 7,$ be the integral in (\ref{trash1}) corresponding to these seven scenarios of (\ref{karr.equation}). The only cases that contribute to a non-zero limit are $I_1, I_6, $ and $I_7$. {For example, if  $h_1 = h_2$,   
\begin{eqnarray} \label{proof.cov.mid}
 I_6 & = &   \frac{m_n  \lambda_n^2}{\nu^4 |S_n|}    \iiiint\limits_{S_n^4}  w_n(h_1 + s_1 - s_2)   w_n(h_2 + s_3 - s_4)     \;    \frac{\tau_m^*(s_1,s_2,s_3,s_4) }{m_n} \;   \nu^2 ds_1 ds_2  \varepsilon_{s_1}(ds_3) \varepsilon_{s_2}(ds_4) \;  \nonumber \\
 & = &  \frac{ \lambda_n^2}{\nu^2 |S_n|}    \iint \limits_{S_n^2}   w_n(h_1 + s_1 - s_2)   w_n(h_1 + s_1 - s_2)  {\tau}_{AB,m}(s_2-s_1)  ds_1  ds_2     \\
	& = &   \frac{ \lambda_n^2}{\nu^2}    \int_{ \frac{S_n - S_n + h_1}{\lambda_n}}    \frac{1}{\lambda_n^2}   w(y)^2    \tau_{AB,m}(h_1 - \lambda_n y) \frac{|S_n \cap (S_n + h_1 - \lambda_n y) |}{|S_n|} dy   \  \nonumber \\
& \rightarrow &  \frac{1}{\nu^2 }  \left( \int_{\mathbb{R}^2} w(y)^2 dy \right)  \tau_{ AB}(h_1)  \nonumber
\end{eqnarray}
by taking $y = \frac{h_1 + s_1 - s_2}{\lambda_n} \; \text{and}\; u =  s_2$ in the last equation. The convergence is from the dominated convergence theorem. On the other hand, if $h_1 \neq h_2$,
\begin{gather*}
I_6 =  \frac{ \lambda_n^2}{\nu^2}    \int_{ \frac{S_n - S_n + h_1}{\lambda_n}}    \frac{1}{\lambda_n^2}   w(y) w \left( y + \frac{h_2 - h_1}{\lambda_n} \right)    \tau_{AB,m}(h_1 - \lambda_n y) \frac{|S_n \cap (S_n + h_1 - \lambda_n y) |}{|S_n|} dy   \rightarrow 0.
\end{gather*}
Similarly, 
\begin{eqnarray} \label{proof.cov.mid.2}
  I_7 \rightarrow \frac{1}{\nu^2 }  \left( \int_{\mathbb{R}^2} w(y)^2 dy \right)  \tau_{ A \cap B A \cap B}(h_1) .
\end{eqnarray}

Turning to $I_1$, we claim
\begin{eqnarray} 
  \label{cond.numerator.cov}  \left|  I_1 - \frac{|S_n| \lambda_n^2}{m_n}  E \left(\hat{\tau}_{AB,m}(h_1) \right)  E \left(\hat{\tau}_{AB,m}(h_2) \right) \right|   \rightarrow 0.
\end{eqnarray}
To see this, observe that the left-hand side in (\ref{cond.numerator.cov}) is  bounded by
\begin{eqnarray*} 
&& \frac{m_n \lambda_n^2}{\nu^4 |S_n|}   \iiiint \limits_{S_n^4}  w_n(h_1 + s_1 - s_2)   w_n(h_2 + s_3 - s_4)  \nonumber \\
&& \qquad \qquad  \qquad \left| \frac{\tau_m^*(\textbf{0}, s_2-s_1,s_3-s_1,s_4-s_1)}{m_n}  -     \frac{{\tau}_{AB,m}(s_2-s_1)}{m_n}  \frac{{\tau}_{AB,m}(s_4-s_3)}{m_n}  \right| \nu^4 ds_1  ds_2 ds_3  ds_4 \nonumber \\
& \leq & \lambda_n^2 m_n\iiint \limits_{(S_n - S_n)^3}  w_n(h_1 - v_1)   w_n(h_2 - (v_3-v_2)) \left|  \frac{ \tau_m^*({\textbf{0}},v_1,v_2, v_3)  }{m_n}  - \frac{{\tau}_{AB,m}(v_1)}{m_n}  \frac{{\tau}_{AB,m}(v_3-v_2)}{m_n}   \right| dv_1  dv_2 dv_3  
\end{eqnarray*}
where the change of variables $ v_1 = s_2 - s_1, v_2 = s_3 - s_1, v_3 = s_4 - s_1 $ are used. By taking $u = v_2, y_1 = \frac{h_1 -  v_1}{\lambda_n}$ and $ y _2 = \frac{h_2 -  (v_3- v_2)}{\lambda_n} $, the right-hand side of the inequality is equivalent to
\begin{eqnarray} \label{Prop.5.7.equation}
&  &   \lambda_n^2 {m_n}   \int_{\frac{(S_n - S_n) - (S_n - S_n) +h_2}{\lambda_n}} \int_{\frac{S_n - S_n+h_1}{\lambda_n}} \int_{S_n - S_n}  w(y_1)   w(y_2)  \nonumber\\
&& \qquad  \left|  \frac{\tau_m^*({\textbf{0}}, h_1 - y_1\lambda_n,u,u + h_2 - y_2\lambda_n)}{m_n}  - \frac{{\tau}_{AB,m}(h_1 - y_1\lambda_n)}{m_n}  \frac{{\tau}_{AB,m}(h_2 - y_2\lambda_n)}{m_n}  \right| du dy_1  dy_2 \nonumber \\
 & =  &    \lambda_n^2 m_n   \; O \left (\int_{\mathbb{R}^2}  \alpha_{2,2}(||y||) dy \right) 
\end{eqnarray}
To see (\ref{Prop.5.7.equation}), observe that $\min d( \{{\textbf{0}},h_1 - y_1\lambda_n\} \{u,u + h_2 - y_2\lambda_n\}) \leq ||u|| + ||u - h_1 + y_1\lambda_n|| +  ||u + h_2 - y_2\lambda_n||+||u + h_2 - y_2\lambda_n\ -h_1 + y_1\lambda_n||.$ Thus, the integral in (\ref{Prop.5.7.equation}) is bounded by 
\begin{eqnarray*}
&& \int_{\mathbb{R}^2} \alpha_{2,2}(||u||) du  \left(\int_{\mathbb{R}^2}  w(y_1)  dy_1 \right)^2 +   \int_{\frac{S_n - S_n+h_1}{\lambda_n}} \int_{S_n - S_n}  w(y_1) \alpha_{2,2}(||u - h_1 + y_1\lambda_n||) du dy_1  \int_{\mathbb{R}^2}   w(y_2) dy_2   \\
&& +     \int_{\frac{S_n - S_n+h_2}{\lambda_n}} \int_{S_n - S_n}  w(y_2) \alpha_{2,2}(||u - h_2 + y_2\lambda_n||) du dy_2  \int_{\mathbb{R}^2}   w(y_1) dy_1  \\
&& +  \int_{\frac{S_n - S_n+h_2}{\lambda_n}}  \int_{\frac{S_n - S_n+h_1}{\lambda_n}}   \int_{S_n - S_n}  w(y_1) w(y_2)  \alpha_{2,2}(||u +  h_2 - h_1 - y_2 \lambda_n + y_1 \lambda_n||) du dy_1 dy_2  \\
&& =A_1 + A_2 + A_3 + A_4.
\end{eqnarray*}
Notice that $ A_1 =   \int_{\mathbb{R}^2} \alpha_{2,2}(||u||) du.$ Take $x = u - h_1 + y_1\lambda_n,$ then
\begin{center}
$ \displaystyle A_2 \leq  \int_{\frac{S_n - S_n+h_1}{\lambda_n}}  \int_{\mathbb{R}^2}  w(y_1)  \alpha_{2,2}(||x||) dx dy_1 \leq \int_{\mathbb{R}^2}   \alpha_{2,2}(||x||) dx   \int_{\mathbb{R}^2}   w(y_1)  dy_1   = \int_{\mathbb{R}^2}  \alpha_{2,2}(||x||) dx.$
\end{center}
Similarly $A_3 \leq \int_{R^2}  \alpha_{2,2}(||x||) dx$ can be shown. Using the similar change of variable technique, 
\begin{center}
$ \displaystyle A_4 \leq \int_{\frac{S_n - S_n+h_2}{\lambda_n}}  \int_{\frac{S_n - S_n+h_1}{\lambda_n}} \int_{S_n - S_n + h_2 - h_1 - y_2 \lambda_n + y_1 \lambda_n}  w(y_1) w(y_2)  \alpha_{2,2}(||x||) dx dy_1 dy_2    \leq  \int_{\mathbb{R}^2}  \alpha_{2,2}(||x||) dx $
\end{center}
Hence, (\ref{Prop.5.7.equation}) is verified, and (\ref{cond.numerator.cov}) is proved. 

Lastly, using the same argument in Lemma A.4. in \citet{lgs}, we have
\begin{center}
$ \displaystyle I_j \rightarrow 0, $ if $ j = 2,3,4,5 $.
\end{center}
Combining the result (\ref{proof.cov.mid})-(\ref{cond.numerator.cov}), (ii) is proved, which completes the proof. }
\end{proof}


Next, we establish the asymptotic normality for $\hat{\tau}_{AB,m}(h)$.
\begin{Proposition} \label{irregular.numerator.2}
Assume that the conditions of Proposition \ref{irregular.numerator.1} and (\textbf{M3}) hold. Then
\begin{center}
 $ \displaystyle \sqrt{ \frac{|S_n|\lambda_n^2}{m_n}} \left(\hat{\tau}_{AB,m}(h)  -  E \hat{\tau}_{AB,m}(h)  \right) \rightarrow  N(0, \sigma^2),$
\end{center}
where  $ \sigma^2= \frac{1}{\nu^2 }  \left( \int_{\mathbb{R}^2} w(y)^2 dy \right)  \tau_{ AB}(h)$. Furthermore,  if $ E \hat{\tau}_{AB,m}(h)  -  \tau_{AB}(h) = o \left( \sqrt{ { \frac{|S_n|\lambda_n^2}{m_n}} }  \right)$, 
\begin{center}
$ \displaystyle \sqrt{ \frac{|S_n|\lambda_n^2}{m_n}} \left(\hat{\tau}_{AB,m}(h)  -  \tau_{AB}(h)  \right) \rightarrow  N(0, \sigma^2).$
\end{center}
\end{Proposition}

\begin{proof} We follow \citet{lgs} with focusing our attention to $\mathbb{R}^2$ and using a classical blocking technique. Let $D_n^i$ be non-overlapping cubes that divide $S_n$ for $i = 1,..., k_n$, where $ k_n=  {|S_n|}/{|D_n^i|}$.  Within each $D_n^i$, $B_n^i$ is an inner cube sharing the same center and $d( \partial D_n^i,B_n^i) \geq n^{\eta}$. Let $ |D_n^i|=n^{2 \alpha}$ and $|B_n^i|= (n^{\alpha}- n^{\eta})^2$ where $6/ (2 + \epsilon)< \eta < \alpha <1$ for some $\epsilon> \frac{2 + 4 \alpha}{\eta}$ . {Let $k_n'$ be the additional number of cubes to cover $S_n$. From  Lemma A.3. in \citet{lgs}, 
\begin{gather}
\label{cond.boundary}  k_n = O(n^{2(1 - \alpha)}) \quad \text{and} \quad k_n' = O(n^{1 - \alpha}). 
\end{gather} }
Now define 
\begin{eqnarray*}
\displaystyle A_n & = & \sqrt{ \frac{m_n   \lambda_n^2}{|S_n| } }  \frac{1}{\nu^2}  \iint \limits_{S_n \times S_n}   w_n(h + s_1 - s_2)  \;  I\left(\frac{X_{s_1}}{a_m} \in A \right)    I\left(\frac{X_{s_2}}{a_m} \in B \right)  N^{(2)}(ds_1, ds_2), \\
 \displaystyle a_{ni}  & = & \sqrt{ \frac{m_n   \lambda_n^2}{ |S_n| } }  \frac{1}{\nu^2}   \iint \limits_{B_n^i \times B_n^i  }   w_n(h + s_1 - s_2)  \;  I\left(\frac{X_{s_1}}{a_m} \in A \right)    I\left(\frac{X_{s_2}}{a_m} \in B \right)  N^{(2)}(ds_1, ds_2), \\
& = & \frac{1}{\sqrt{k_n}} \sqrt{ \frac{m_n   \lambda_n^2}{|D_n^i| }}  \frac{1}{\nu^2}  \iint \limits_{B_n^i \times B_n^i }  w_n(h + s_1 - s_2)  \;  I\left(\frac{X_{s_1}}{a_m} \in A \right)    I\left(\frac{X_{s_2}}{a_m} \in B \right)  N^{(2)}(ds_1, ds_2), \\
\tilde{A}_n & = & {A_n} - E {A_n}, \quad \tilde{a}_{ni} = a_{ni} - E a_{ni},  \quad   {a}_n =  \sum_{i=1}^{k_n}  {a}_{ni}, \quad  \tilde{a}_n =  \sum_{i=1}^{k_n} \tilde{a}_{ni}, \quad \tilde{a}_n'= \sum_{i=1}^{k_n}  \tilde{a}_{ni}',
\end{eqnarray*}
where $\tilde{a}_{ni}'$ denotes an independent copy of $\tilde{a}_{ni}$. \\

\textbf{Step 1}. Show $  \text{var}(\tilde{A}_n - \tilde{a}_n)\rightarrow 0$. \\
{We will prove Step 1 by showing:

 i) $\text{var}(\tilde{A}_n) \rightarrow \frac{1}{\nu^2 }  \left( \int_{\mathbb{R}^2} w(y)^2 dy \right)  \tau_{ AB}(h)$,

 ii) $\text{cov}(\tilde{A}_n, \tilde{a}_n)  \rightarrow \frac{1}{\nu^2 }  \left( \int_{\mathbb{R}^2} w(y)^2 dy \right)  \tau_{ AB}(h)$, and

iii) $\text{var}(\tilde{a}_n)  \rightarrow \frac{1}{\nu^2 }  \left( \int_{\mathbb{R}^2} w(y)^2 dy \right)  \tau_{ AB}(h)$.\\ 
\ \\
 i) This follows from Proposition \ref{irregular.numerator.1} (iii). \\
\ \\
ii) Recall $ \tau_m^*(s_1,s_2,s_3,s_4) $ defined in Proposition \ref{irregular.numerator.1} (ii). Then 
 \begin{eqnarray*}
&& E \left(    {A}_n    a_n \right) \\
& = &   \frac{\lambda_n^2}{\nu^4 |S_n|}  \sum_{i=1}^{k_n} \; \iiiint \limits_{S_n \times S_n \times B_n^i  \times B_n^i}      w_n(h + s_1 - s_2)  \; w_n(h + s_3 - s_4)   \tau_m^*(s_1,s_2,s_3,s_4)   E[ N^{(2)}(ds_1, ds_2)      N^{(2)}(ds_3, ds_4) ] \\
& = &  \frac{\lambda_n^2}{\nu^4 |S_n|}  \sum_{i=1}^{k_n} \;   \left[ \quad   \iiiint \limits_{S_n \setminus  B_n^i \times S_n \setminus  B_n^i \times B_n^i \times B_n^i}    \cdot +   \iiiint \limits_{S_n \setminus  B_n^i \times   B_n^i \times B_n^i \times B_n^i }    \cdot    +    \iiiint \limits_{ B_n^i \times S_n \setminus  B_n^i  \times B_n^i \times B_n^i }    \cdot  +   \iiiint \limits_{(B_n^i)^4}    \cdot  \right] \\
& = &   D_1 + D_2 + D_3 + D_4 \\
& = & \sum_{i = 1}^4 \sum_{j=1}^7 D_i^j
\end{eqnarray*} 
where $ D_i^j$ be the integral in $D_i$ corresponding to the seven cases of $ E[ N^{(2)}(ds_1, ds_2)      N^{(2)}(ds_3, ds_4) ]$  as in (\ref{karr.equation}) for $i = 1,...,4$ and $j = 1,..., 7$. As shown in the proof of Proposition \ref{irregular.numerator.1} (ii), non-zero contributions only arise  when $j = 1,6,$ and $7$. By the similar arguments in (\ref{cond.numerator.cov}), 
\begin{center}
$|\sum_{i = 1}^4  D_i^1  - E({A}_n)E({a}_n)  | \rightarrow 0$.
\end{center}
Since $j=6 $ and $7$ only occur when $s_1, s_2, s_3,s_4 \in B_n^i$, we only consider $D_4^6 + D_4^7$ which equals to  
 \begin{eqnarray*}
& & \frac{\lambda_n^2}{\nu^4 |S_n|}  \sum_{i=1}^{k_n} \;     \iint \limits_{B_n^i \times B_n^i }       \left[ w_n(h + s_1 - s_2)^2    +  w_n(h + s_1 - s_2)  w_n(h + s_2 - s_1)    \right] {\tau}_{AB,m}(s_2 - s_1)  {\nu^2} ds_1 ds_2     \\
& = &    \frac{m_n  \lambda_n^2}{\nu^2 |D_n^1|}   \iint \limits_{ B_n^1 \times B_n^1}       \left[ w_n(h + s_1 - s_2)^2    +  w_n(h + s_1 - s_2)  w_n(h + s_2 - s_1)    \right] {\tau}_{AB,m}(s_2 - s_1)    ds_1 ds_2     \\
& \rightarrow & \frac{1}{v^2 } \int_{\mathbb{R}^2}   w({y})^2   dy  \;  \tau_{AB}(h).
\end{eqnarray*} 
The convergence is derived from arguments in (\ref{proof.cov.mid}) and (\ref{cond.numerator.cov}). Thus, we conclude
\begin{gather*}
\text{cov}(\tilde{A}_n, \tilde{a}_n)= \left( \sum_{i = 1}^4 \sum_{j=1}^7 D_i^j \right) - E({A}_n)E({a}_n) =  D_4^6 + D_4^7 +   o(1)  \rightarrow \frac{1}{\nu^2 }  \left( \int_{\mathbb{R}^2} w(y)^2 dy \right)  \tau_{ AB}(h).
\end{gather*}
\ \\
 iii) Let $   \text{var}(\tilde{a}_n) = \sum_{ i=1}^{ k_n}   \text{var}(\tilde{a}_{ni})+ \sum_{1 \leq i \neq j \leq k_n}  \text{cov}(\tilde{a}_{ni} ,\tilde{a}_{nj}) $.  Note from Proposition \ref{irregular.numerator.1}  (iii) that
\begin{center}
	$\displaystyle \sum_{i=1}^{k_n}   \text{var}(\tilde{a}_{ni}) =  k_n \text{var}( {a}_{n1})   \rightarrow \frac{1}{\nu^2 }  \left( \int_{\mathbb{R}^2} w(y)^2 dy \right)  \tau_{ AB}(h).$  
\end{center}


Also note that since $\tilde{a}_{ni}$ and $\tilde{a}_{nj}$ are integrals over disjoint sets for $i \neq j$ and $X_s$ is independent of $N$, $E [  \tilde{a}_{ni}|N ] $ and $E [\tilde{a}_{nj}|N ] $ are independent. Thus, 
 \begin{eqnarray*}
\sum_{1 \leq i \neq j \leq k_n} |   \text{cov} (\tilde{a}_{ni}, \tilde{a}_{nj})  | & = & \sum_{1 \leq i \neq j \leq k_n} |   E \{ \text{cov} (\tilde{a}_{ni}, \tilde{a}_{nj}|N) \}   + \text{cov} \{ E(\tilde{a}_{ni}|N) ,   E(\tilde{a}_{nj}|N)  \} |  \\
& = & \sum_{1 \leq i \neq j \leq k_n} |   E \{ \text{cov} (\tilde{a}_{ni}, \tilde{a}_{nj}|N) \}  |.  
\end{eqnarray*}
Notice from  Proposition \ref{covar.upper} and $|   a_{ni} |   \leq   {\sqrt{  \frac{m_n   \lambda_n^2}{ |S_n| } }}{ |B_n^i|}$ that 
 \begin{eqnarray*}
 E \{ \text{cov} (\tilde{a}_{ni}, \tilde{a}_{nj}|N) \} 	& \leq &    const  \;     \frac{m_n   \lambda_n^2}{ |S_n| }  |B_n^i|  |B_n^j| \ \; |   E ( \alpha_{M,M } (n^{\eta})|N)  |   \leq const  \;     \frac{m_n   \lambda_n^2}{ |S_n| }  |B_n^1|^2  \;E (M^2)   n^{ - \epsilon \eta}  \\
\end{eqnarray*}
where $ M =  \max\{N(B_n^i),N(B_n^j)\}$ and the last inequality is from  (\ref{clt.cond.3}). Since $ k_n = |S_n|/ |D_n^1|$ where $ |S_n| = n^2, |D_n^1| = n^{2 \alpha},  |B_n^1| = O(n^{2 \alpha}) $, 
 \begin{eqnarray*}
\sum_{1 \leq i \neq j \leq k_n} |   \text{cov} (\tilde{a}_{ni}, \tilde{a}_{nj})  |     
 \leq   const  \;   k_n^2    \frac{m_n   \lambda_n^2}{ |S_n| }   { |B_n^1|^2}   \; |B_n^1|^2   n^{ - \epsilon \eta} 
=  O \left( m_n   \lambda_n^2    n^{2+ 4 \alpha  - \epsilon \eta} \right) 
\end{eqnarray*}
which converges to 0 as $m_n   \lambda_n^2 \rightarrow 0 $ and $\epsilon> \frac{2 + 4 \alpha}{\eta}$.} \\


\textbf{Step 2}. Show {$| \phi_n(x) - \phi_n'(x) | \rightarrow 0$} where $\phi_n(x)$ and $\phi_n'(x)$ are the characteristic functions of $\tilde{a}_n$ and $\tilde{a}_n'$. 

Analogously to the idea presented in (6.2) in \citet{Richard09}, 
\begin{center}
$\displaystyle	| \phi_n(x) - \phi_n'(x) |  
= \left|    \sum_{l = 1}^{k_n} E \prod_{j=1}^{l-1} e^{i x \frac{\tilde{a}_{nj}}{\sqrt{k_n}}} \left(e^{i x \frac{\tilde{a}_{nl}}{\sqrt{k_n}}} - e^{i x \frac{\tilde{a}'_{nl}}{\sqrt{k_n}}} \right) \prod_{j=l+1}^{k_n} e^{i x \frac{\tilde{a}'_{nj}}{\sqrt{k_n}}} \right|  \leq  \sum_{l = 1}^{k_n} \left| \text{cov} \left( \prod_{j=1}^{l-1} e^{i x \frac{\tilde{a}_{nj} }{\sqrt{k_n}}} , e^{i x \frac{\tilde{a}_{nl}}{\sqrt{k_n}}} \right) \right| $
\end{center}
Using the same technique in Step 1 iii), 
\begin{center}
$   \left| \text{cov} \left( \prod_{j=1}^{l-1} e^{i x \frac{\tilde{a}_{nj} }{\sqrt{k_n}}} , e^{i x \frac{\tilde{a}_{nl}}{\sqrt{k_n}}} \right) \right|  \leq     {const} \; E  \left( \alpha_{M,M} (n^{\eta}) \right)  
	 \leq   {const} \;   E \left( M  ^2 \right)  n^{ -\epsilon \eta}  
	 \leq  {const} \;    l^2 | n^{ 2\alpha } |^2 n^{ -\epsilon \eta}  $
\end{center}
where  $M = N( \cup_{j = 1}^{l} B_{n}^j).$ The second and the last inequality is from (\ref{clt.cond.3}) and $|B_n^1|=O(n^{2 \alpha})$ respectively. Hence, from $k_n = n^{2 - 2 \alpha}$, we have
\begin{center}
$ \displaystyle | \phi_n(x) - \phi_n'(x) |   \leq  const \; \sum_{l = 1}^{k_n}   l^2 | n^{ 2\alpha } |^2 n^{ -\epsilon \eta}   \leq O (n^{6 - 2 \alpha - \epsilon \eta})$
\end{center}
which converges to 0 from $6/ (2 + \epsilon) < \eta < \alpha <1.$ \\

\textbf{Step 3}. Show the central limit theorem holds for $\tilde{a}_n' $. 

Let $ I_{ni} = \int_{B_n^i} \int_{B_n^i} w_n(h + s_1 - s_2)  \;  I\left(\frac{X_{s_1}}{a_m} \in A \right)    I\left(\frac{X_{s_2}}{a_m} \in B \right) N^{(2)}(ds_1, ds_2).$ By (\ref{clt.cond.2}), we have
\begin{eqnarray*}
 E | \sqrt{k_n}  \tilde{a}_{ni}' |^{2 + \delta}  &=&  E \left| \sqrt{ \frac{m_n   \lambda_n^2}{|D_n^i| }}  \frac{1}{\nu^2}  \left[ I_{ni}  - E(I_{ni} ) \right] \right|^{2 + \delta}  \\ 
&=&  E \left| \sqrt{ \frac{|B_n^i|^2  \lambda_n^2}{|D_n^i|   m_n}}   \left[ \hat{\tau}_{AB,m}(h:B_n^i)  - E(\hat{\tau}_{AB,m}(h:B_n^i) ) \right] \right|^{2 + \delta}   <  C_{\delta} 
\end{eqnarray*}
As $(\tilde{a}_{ni}')$ is triangular array of independent random variables with $\text{var}(\sum_{i = 1}^{k_n} \tilde{a}_{ni}' )  = \sigma_n^2 \rightarrow \sigma^2$, and
\begin{gather*}
 \frac{   \sum_{i=1}^{k_n} E|\tilde{a}_{ni}'|^{2+\delta}}{    (  \sigma_n )^{2+\delta}   }  \leq     \frac{ k_n   k_n^{-(1+\delta/2)} C_{\delta} }{  (  \sigma_n )^{2+\delta} } \rightarrow 0,
\end{gather*}
Lyapunov's condition is satisfied and hence the central limit theorem holds. \end{proof} 

\begin{proof}  [Proof of Theorem \ref{theorem:clt:2}]

Proposition \ref{irregular.denominator} implies $\hat{p}_m(A) \xrightarrow{p} \mu(A)$. By Slutsky's theorem and Proposition \ref{irregular.numerator.2},
\begin{eqnarray*}
\sqrt{\frac{|S_n|\lambda_n^2}{m_n}  } \left( \frac{\hat{\tau}_{AB,m}(h)}{\hat{p}_m(A)} - \frac{{\tau}_{AB,m}(h)}{\hat{p}_m(A)} \right)  = \sqrt{\frac{|S_n|\lambda_n^2}{m_n}  } \left( \hat{\rho}_{AB,m}(h)  - \frac{{\tau}_{AB,m}(h)}{\hat{p}_m(A)} \right) & \rightarrow&  N(0, {\sigma}^2/\mu(A)^2).
\end{eqnarray*}
Recall from Proposition \ref{irregular.denominator} that $ \displaystyle \text{var}({\hat{p}_m(A)} ) = O\left( {m_n} / {|S_n|} \right)$. Then
\begin{eqnarray*}
\sqrt{\frac{|S_n|\lambda_n^2}{m_n}  } \left( \hat{\rho}_{AB,m}(h)  - \frac{{\tau}_{AB,m}(h)}{\hat{p}_m(A)} \right)    = \sqrt{\frac{|S_n|\lambda_n^2}{m_n}  } \left( \hat{\rho}_{AB,m}(h) - {\rho}_{AB,m}(h) \right) +o_p(1).
\end{eqnarray*}
Thus, the central limit theorem for $ \sqrt{\frac{|S_n|\lambda_n^2}{m_n}  } \left( \hat{\rho}_{AB,m}(h) - {\rho}_{AB,m}(h) \right) $ is proved. The joint normality (\ref{final}) is established using the Cram\'{e}r-Wold device. \end{proof}

\subsection{Appendix C:  Example \ref{example.br.2}} \label{appendix C}

First, we show that $X_s$ satisfies LUNC in (\ref{note.2}). Notice that the process has continuous sample paths a.s. since the Gaussian process $\{W_s - \delta(s), s \in \mathbb{R}^2\}$ in (\ref{Def.MS})  has continuous sample paths. Notice from \citet{lindgren(2012)}, Section 2.2, that a Gaussian process with a continuous correlation function satisfying (\ref{cond.corr.BR}) has continuous sample paths.

From (\ref{Def.MS}), let $X_s = U_s^1 \vee U_s^2$, where $U_s^1 = \Gamma_1^{-1} Y_s^1 $ and $U_s^2 = \sup_{j \geq 2} \Gamma_j^{-1} Y^j_s$. Then 
\begin{eqnarray*}
nP \left( \sup_{||s|| < \delta'} \frac{|X_s - X_{\textbf{0}}|}{a_n} > \delta  \right) & = &   nP \left( \sup_{||s|| < \delta'}  |U_s^1 \vee U_s^2 - U_{\textbf{0}}^1 \vee U_{\textbf{0}}^2 |  > a_n \delta  \right) \\
& \leq &  n P \left( \sup_{||s|| < \delta'}  |U_s^1 - U_{\textbf{0}}^1 |  > \frac{ a_n \delta}{2}  \right) + n P \left( \sup_{||s|| < \delta'}  | U_s^2 - U_{\textbf{0}}^2 |  > \frac{a_n \delta}{2}  \right)\\
 & =& A_1 + A_2.
\end{eqnarray*}
Since $E | \sup_{||s|| < \delta'} {|Y(s)|} | <  \infty$ (see Proposition 13 in \citet{kabluchko}), we can apply the dominated convergence theorem to obtain
\begin{eqnarray*}
A_1 = nP \left( \Gamma_1 <  \frac{2 \sup_{||s|| < \delta'} {|Y_s^1 - Y_{\textbf{0}}^1|}}{\delta a_n}  \right)  & = & n \int   \left(  1 - e^{-z/ \delta a_n } \right)  g(Z)   dZ \rightarrow\frac{2 E(\sup_{||s|| < \delta'} {|Y_s - Y_{\textbf{0}}|})}{\delta}  \rightarrow  0, 
\end{eqnarray*}
where $ Z= 2 \sup_{||s|| < \delta'} {|Y_s - Y_{\textbf{0}}|} $. 

To show $A_2 \rightarrow 0$, we follow the arguments in \citet{Richard08}. 
\begin{eqnarray*}
A_2 =   nP \left( \sup_{||s|| < \delta'} \bigvee_{j \geq 2}^{\infty} \Gamma_j^{-1}  {|Y^j_s - Y^j_s|} > \frac{a_n \delta}{2}  \right) 
& \leq & n \sum_{j = 2}^{\infty} P \left(   2 {\sup_{||s|| < \delta'} {|Y_s|}} > \Gamma_j  {\delta a_n}/2  \right) \\
& = & n \int   \left(  \sum_{j \geq 2}^{\infty}  P \left(  4  y > \Gamma_j  {\delta a_n}  \right)   \right)  P \left(  {\sup_{||s|| < \delta'} {|Y_s|}}  \in dy \right)\\
& = & n \int_{0}^{\infty}   \left( \frac{4y}{\delta a_n} - \left(1 - e^{-\frac{4y}{\delta a_n}} \right)    \right)  P \left(  {\sup_{||s|| < \delta'} {|Y_s|}}  \in dy \right) 
\end{eqnarray*}
{The last line is from $ E T[0, \frac{4y}{\delta a_n}] = \sum_{j=1}^{\infty} P \left( \Gamma_j  < \frac{ 4y}{\delta a_n} \right)  = \frac{4y}{\delta a_n},$ where $ T =  \sum_{j=1}^{\infty} \epsilon_{\Gamma_j}$ is a homogeneous point process.} The dominated convergence theorem applies as $f_n(y) = n \left(  \frac{4y}{\delta a_n} - (1 - e^{-\frac{4y}{\delta a_n}} ) \right) \leq cy $ for some $c > 0 $, all $y >0$ and $f_n(y) \rightarrow 0$ as $n \rightarrow 0$, and $E  {\sup_{||s|| < \delta'} {|Y_s|}} < \infty$ from \citet{kabluchko}.

Now we check conditions (\ref{M2})-(\ref{clt.cond.3}). {Recall from (\ref{BR.alpha.bound}) that $\alpha_{c,  c} (h)   \leq    const \;  \frac{1}{\sqrt{||h||^{\alpha}}}  e^{- {\theta ||h||^{\alpha}}/2} $ holds for the process. For convenience in the calculations that follow, set $g(h) =  \frac{1}{\sqrt{||h||^{\alpha}}}  e^{- { \theta ||h||^{\alpha}}/2}$. We will find the sufficient conditions for  (\ref{M2})-(\ref{clt.cond.3}). For (\ref{M2}),
\begin{eqnarray}
&& \label{M4'}  \int_{\mathbb{R}^2} g(y) dy < \infty
\end{eqnarray}
is sufficient. To see this, infer from (\ref{BR.calculation}) that
\begin{eqnarray*}
m_n     P(X_y > \epsilon a_m, X_{\textbf{0}} > \epsilon a_m) = m_n \left[ 1 - 2 e^{- 1/a_m} + e^{- 2 \Phi(\sqrt{\delta(h)})/a_m} \right] = \frac{2 m_n}{a_m}\bar{\Phi}(\sqrt{\delta(h)})  + O \left(\frac{m_n}{a_m^2}\right).
\end{eqnarray*}
Thus
\begin{eqnarray*}
m_n  \int_{B[k,r_n]}  P(X_y > \epsilon a_m, X_{\textbf{0}} > \epsilon a_m)dy 
& = &  \int_{B[k,r_n]} \frac{2 m_n}{a_m}  \bar{\Phi}( \sqrt{\delta(y)}) dy  + O \left( \frac{r_n^2}{m_n} \right)   \\
& \leq & const \; \int_{B[k,\infty]} g( y) dy + o(1), 
\end{eqnarray*}
where the last inequality is from (\ref{BR.alpha.bound}). 

From (\ref{BR.alpha.bound}), the condition (\ref{M1}) is satisfied if 
\begin{eqnarray}
&& \label{M1'}      \int_{\mathbb{R}^2 \setminus B[0,r_n)}  m_n  g(y) dy \rightarrow  0.
\end{eqnarray}
Similarly, using (\ref{BR.alpha.bound}), the second condition in (\ref{irregular.numerator.integration}) is implied if (\ref{M4'}) holds.
The condition  (\ref{clt.cond.3}) is checked immediately from (\ref{dombry}) since
\begin{gather*}
\label{M5'} \sup_l \frac{\alpha_{l, l} (||h||)}{l^2}   \leq    const \; \frac{1}{\sqrt{||h||^{\alpha}}}  e^{- {\theta ||h||^{\alpha}}/2}  = O(||h||^{- \epsilon}).
\end{gather*}
We check the condition (\ref{clt.cond.2}) with $\delta = 1$ is satisfied if (\ref{M000}) assumed, but we skip this as it is tedious. Hence, it suffices to find conditions under which (\ref{M4'}) and (\ref{M1'}) hold.  

\begin{remark}
\YB{If the process is regularly varying in the space of continuous functions in every compact set, then LUNC is satisfied. See \citet{Hult(2006)}, Theorem 4.4.} 
\end{remark}

\begin{Proposition} For  Example \ref{example.br.2}, the conditions (\ref{M4'}) - (\ref{M1'}) hold if $ \log m_n = o(r_n^a)$.  
\end{Proposition} 

\begin{proof} 
Using change of variables to polar coordinates and $r^a /8 =t$,  (\ref{M4'}) is checked. For $a \in (0,2]$
\begin{eqnarray*}
 \int_{\mathbb{R}^2} g(y) dy 
= const \;   \int_{0}^{\infty} t^{\frac{2}{a} - \frac{3}{2}}e^{ - t} dt  < \infty.
\end{eqnarray*}
 For (\ref{M1'}), notice that for sufficiently large $n, \;  m_n  g(r_n) \leq m_n  e^{- \theta r_n^{\alpha}/2}  = o(1)$ provided $ \displaystyle \log m_n = o(r_n^a)$. This completes the proof.
\end{proof}}

Finally, we find the condition under which (\ref{replace}) holds.
\begin{Proposition} For the Brown-Resnick process, (\ref{replace}) holds if $\frac{|S_n| \lambda_n^2}{m_n^3} \rightarrow 0$. \end{Proposition}

\begin{proof} From (\ref{BR.calculation}),
\begin{center}
$ \displaystyle |{\rho}_{AB,m}(h) - {\rho}_{AB}(h)| =\frac{1+o(1)}{\mu(h)} |{\tau}_{AB,m}(h) \mu(h) - {\tau}_{AB}(h) p_m(h)| = \frac{1+o(1)}{\mu(h)}   O({m_n}/{a_m^2}) =  O({1}/{m_n}). $
\end{center}
Therefore, (\ref{replace}) holds if $\frac{|S_n| \lambda_n^2}{m_n^3} \rightarrow 0$.
\end{proof}

\section*{Acknowledgments}
We thank Christina Steinkohl for discussions on the proofs and suggestions for the application. We also would like to thank Chin Man (Bill) Mok for providing the Florida rainfall data and acknowledge that the data is provided by the Southwest Florida Water Management District (SWFWMD). All authors acknowledge the support by the TUM Institute for Advanced Study(TUM-IAS). The second author's research was partly supported by ARO MURI grant W11NF-12-1-0385.


\bibliographystyle{IEEEtranN}
\bibliography{reference}
\end{document}